\documentclass{amsart}
\usepackage{setspace}
\usepackage{a4}
\usepackage{amssymb,amsmath,amsthm,latexsym}
\usepackage{amsfonts}
\usepackage{amsfonts}
\usepackage{graphicx}
\usepackage{textcomp}
\usepackage{cite}
\newtheorem{theorem}{Theorem}[section]

\newtheorem{corollary}[theorem] {Corollary}
\newtheorem{definition}[theorem]{Definition}
\newtheorem{example}[theorem]{Example}
\newtheorem{lemma} [theorem]{Lemma}

\newtheorem{proposition}[theorem]{Proposition}

\title{This is the title}
\usepackage{amssymb}
\usepackage{amssymb}
\usepackage{amssymb}
\usepackage{amssymb}
\usepackage{amsmath}
\usepackage{tikz}
\usepackage{hyperref}
\usepackage{enumerate}
\usepackage{mathtools}
\usepackage{amsmath}
\usepackage{tikz}
\usepackage{graphicx}
\usepackage{textcomp}
\usetikzlibrary{cd}
\begin{document}
	\begin{center}
{\bf{FRAMES  FOR   METRIC  SPACES}}\\
K. MAHESH KRISHNA AND P. SAM JOHNSON  \\
Department of Mathematical and Computational Sciences\\ 
National Institute of Technology Karnataka (NITK), Surathkal\\
Mangaluru 575 025, India  \\
Emails: kmaheshak@gmail.com,  sam@nitk.edu.in\\

Date: \today\\

\end{center}

\hrule
\vspace{0.5cm}
%--------------------------------------
\textbf{Abstract}: We  make a systematic study of frames for metric spaces. We prove that every  separable metric space admits a metric $\mathcal{M}_d$-frame. Through Lipschitz-free Banach spaces we show that there is a correspondence
between  frames for metric spaces  and  frames for subsets of  Banach spaces. We derive some characterizations of metric frames. 
 We also derive stability results for metric frames.



%\tableofcontents

\section{Introduction}

  Materialized  from the prime paper of Duffin and Schaeffer  in 1952 \cite{DUFFINSCHAEFFER} and stimulated from the
  significant paper of   Daubechies, Grossmann,  and Meyer  in 1986  \cite{DAUBECHIESGROSSMANMEYER} the theory of
  frames for Hilbert spaces  began as a discipline in Mathematics. Power of frames is that it allows to construct every element of the Hilbert space (in more than one way) and it   captures the geometry of Hilbert space. Today, this theory finds its appearance and uses in diverse theoretical and applied
  areas.     We start by making a review of the notion 
   ``frame" in Hilbert and Banach spaces.
  \begin{definition}\cite{DUFFINSCHAEFFER}\label{FREMEDEF}
  	A collection $\{\tau_n\}_n$ in a Hilbert space $\mathcal{H}$ is said to be a
  	frame for $\mathcal{H}$ if there exist $a,b>0$ such that 
  	\begin{align*}
  		a\|h\|^2 \leq \sum_{n=1}^\infty |\langle h, \tau_n\rangle|^2\leq b\|h\|^2, \quad \forall h \in \mathcal{H}.
  	\end{align*}
  	Constants $a$ and $b$ are called as lower and upper frame bounds, respectively. If $a$ can take the   value 0, then we say 
  	$\{\tau_n\}_n$ is a Bessel sequence for  $\mathcal{H}$.
  \end{definition}
  Simple Definition \ref{FREMEDEF} gives the following fundamental result. 
  \begin{theorem}\cite{CASAZZA1997, CHRISTENSENBOOK}\label{FUNDAMENTALTHEOREM}
  	Let $\{\tau_n\}_n$ be a frame for $\mathcal{H}$ with bounds  $a$ and $b$. Then
  	\begin{enumerate}[\upshape(i)]
  	%	\item $\overline{\operatorname{span}}\{\tau_n\}_n=\mathcal{H}$.
  	%	\item $\|\tau_n\|\leq \sqrt{b}, \forall n \in \mathbb{N}$.
  		\item The frame operator  $S_\tau :\mathcal{H} \ni h \mapsto \sum_{n=1}^\infty \langle h, \tau_n\rangle\tau_n\in	\mathcal{H}$
  		is  well-defined  bounded linear, positive and invertible. Further, $\sqrt{a}\|h\|\leq \|\sum_{n=1}^\infty \langle h, \tau_n\rangle\tau_n\|\leq \sqrt{b} \|h\|, \forall h \in \mathcal{H}$ and 
  		\begin{align*}
  			h=\sum_{n=1}^\infty \langle h, S_\tau^{-1}\tau_n\rangle \tau_n=\sum_{n=1}^\infty
  			\langle h, \tau_n\rangle S_\tau^{-1}\tau_n, \quad \forall h \in
  			\mathcal{H}.
  		\end{align*}
  	\item $\{S_\tau^{-1}\tau_n\}_n$ is a frame for $\mathcal{H}$.
  		\item The analysis operator $
  			\theta_\tau: \mathcal{H} \ni h \mapsto \{\langle h, \tau_n\rangle \}_n \in \ell^2(\mathbb{N})$
  		is  well-defined  bounded linear, injective operator. Moreover, $\sqrt{a}\|h\|\leq \|\theta_\tau h\|\leq \sqrt{b} \|h\|, \forall h \in \mathcal{H}$ and the  adjoint of $\theta_\tau$ (synthesis operator) is given by  $\theta_\tau^*: \ell^2(\mathbb{N}) \ni \{a_n \}_n \mapsto \sum_{n=1}^\infty a_n\tau_n \in \mathcal{H}$
  			which is surjective.
  		\item Frame operator 
  		factors as $S_\tau=\theta_\tau^*\theta_\tau.$
  		\item  Let  $\{\omega_n\}_n$ be a collection in  $\mathcal{H}$ satisfying the following:
  		There exist $\alpha, \beta, \gamma \geq 0$ with $\operatorname{max} \{\alpha +\frac{\gamma}{\sqrt{a}},\beta\}<1$ and for $ m=1,2, \dots$, 
  		\begin{align*}
  			\left\|\sum_{n=1}^mc_n(\tau_n-\omega_n)\right\|\leq \alpha \left\|\sum_{n=1}^mc_n\tau_n\right\|+
  			\beta\left\|\sum_{n=1}^mc_n\omega_n\right\|+\gamma \left( \sum_{n=1}^m|c_n|^2\right)^\frac{1}{2}, \quad \forall c_1, \dots, c_m\in 
  			\mathbb{K}.
  		\end{align*}
  		Then $\{\omega_n\}_n$ is a frame for   $\mathcal{H}$ with bounds 
  	$
  			a \left(1-\frac{\alpha+\beta+\frac{\gamma}{\sqrt{a}}}{1+\beta}\right)^2\text{~and~}
  			b\left(1+\frac{\alpha+\beta+\frac{\gamma}{\sqrt{a}}}{1+\beta}\right)^2.
  		$
  		
  		%\item If a vector can be written as $h=\sum_{n=1}^\infty a_n\tau_n$ for some sequence $\{a_n\}_n$  of scalars, then 
  		%\begin{align*}
  		%	\sum_{n=1}^\infty|a_n|^2=\sum_{n=1}^\infty|\langle h, S_\tau^{-1}\tau_n\rangle|^2+\sum_{n=1}^\infty|h-\langle h, S_\tau^{-1}\tau_n\rangle|^2.
  		%\end{align*}
  		
  		\item $ \theta_\tau S_\tau^{-1} \theta_\tau^*$ is a projection onto $ \theta_\tau(\mathcal{H})$.
  	%	\item $\|S_\tau^{-1}\|^{-1}$ is a lower frame bound and $\|S_\tau\|$ is an upper frame bound.
  	\end{enumerate}
  \end{theorem}

  Theorem \ref{FUNDAMENTALTHEOREM} says several things. First, it says that every vector in the Hilbert
  space admits an expansion, called as generalized Fourier expansion, similar to
  Fourier
  expansion coming from an orthonormal basis for a Hilbert space. Second, it
  says that coefficients in the expansion of a vector need not be unique. This is
  particularly important in applications, since loss in the information of a
  vector is less if some of the coefficients are missing. Third, given a frame,
  it naturally generates other frames. Fourth, a frame gives bounded linear
  injective operator from the less known inner product on the Hilbert space to
  the well known standard inner product on the standard separable Hilbert space $\ell^2(\mathbb{N})$.
  Frame inequality now clearly says that there is a comparision of norms between $\mathcal{H}$ and
  the standard Hilbert space $\ell^2(\mathbb{N})$. Fifth, a frame  embeds $\mathcal{H}$ in $\ell^2(\mathbb{N})$
   through the bounded linear  operator $\theta_\tau$. Sixth,  whenever a Hilbert 
  space admits a frame it becomes an image of a surjective operator $\theta_\tau^*$ from
  the $\ell^2(\mathbb{N})$ to it. Seventh, frames are stable w.r.t. perturbations i.e., a small
  change in frame also yields a frame. For further information on frame theory in Hilbert spaces
  we refer the reader to \cite{HANLARSONMEMO,   HEILBOOK, ALDROUBIPORTRAITS,  CASAZZABOOK, KORNELSON, BALAN} .\\

  A look at Definition \ref{FREMEDEF} gives the  natural question: How do we set Definition \ref{FREMEDEF}
  for Banach spaces? It is important to keep in mind that it is more important
  to get results  (at least Fourier like) similar to Theorem \ref{FUNDAMENTALTHEOREM} than defining  a
  generalization of Definition \ref{FREMEDEF} in Banach spaces. This was first done by
  Grochenig in 1991 \cite{GROCHENIG}. After the study of several function spaces \cite{FEICHTINGERCHOOSE},  Grochenig first studied the notion of an atomic decomposition for
  Banach spaces and then defined the notion of a Banach frame. For stating these definitions we want to know the definition 
  of BK-space (Banach scalar valued sequence space or Banach coordinate space).
  \begin{definition} \cite{BANASMURSALEEN}
  	A sequence space $\mathcal{X}_d$ is said to be  a BK-space if all the coordinate functionals are continuous, i.e.,
  	whenever $\{x_n\}_n$ is a sequence in $\mathcal{X}_d$ converging to $x \in \mathcal{X}_d$, then each coordinate of 
  	$x_n$ converges to each coordinate of $x$.
  \end{definition}
   Familiar sequence spaces like  $\ell^p(\mathbb{N})$, $c(\mathbb{N})$ (space of convergent sequences) and $c_0(\mathbb{N})$ (space of  sequences converging to zero) are examples of  BK-spaces. 
 \begin{definition}\cite{GROCHENIG}\label{ATOMICDECOMPODEFI}
 	Let $\mathcal{X}$ be a Banach space and $\mathcal{X}_d$  be  an associated  BK-space. Let $\{f_n\}_n$ be a collection in $\mathcal{X}^*$ and 
 	$\{\tau_n\}_n$ be a collection in $\mathcal{X}$. The pair $(\{f_n\}_n, \{\tau_n\}_n)$ is said to be an atomic decomposition for $\mathcal{X}$ w.r.t $\mathcal{X}_d$ 
 	if the following holds.
 	\begin{enumerate}[\upshape(i)]
 		\item $\{f_n(x)\}_n \in \mathcal{X}_d$, for  each $x \in \mathcal{X}$. 
 		\item There exist $a,b>0$ such that 
 		$
 			a\|x\|\leq \|\{f_n(x)\}_n\|\leq b\|x\|,  \forall x \in \mathcal{X}.
 	$
 		\item $x=\sum_{n=1}^\infty f_n(x)\tau_n$, for  each $x \in \mathcal{X}$.
 	\end{enumerate}
 	Constants $a$ and $b$ are called as atomic  bounds for the atomic decomposition.
 \end{definition}
 From Theorem \ref{FUNDAMENTALTHEOREM} one sees that whenever  $\{\tau_n\}_n$ is a frame for $\mathcal{H}$, then by defining $f_n (h)\coloneqq \langle h, S_\tau^{-1} \tau_n\rangle $, $ \forall h \in \mathcal{H}$, $ \forall n$, the pair $(\{f_n\}_n, \{\tau_n\}_n)$ satisfies all the conditions of  Definition \ref{ATOMICDECOMPODEFI} and hence it is an atomic decomposition for $\mathcal{H}$ w.r.t $\ell^2(\mathbb{N})$.
  
  \begin{definition}\cite{GROCHENIG}\label{BANACHFRAMEDEF}
  	Let $\mathcal{X}$ be a Banach space and $\mathcal{X}_d$  be  an associated  BK-space. Let $\{f_n\}_n$ be a collection in $\mathcal{X}^*$ and $S:\mathcal{X}_d \to \mathcal{X}$ 
  	be a bounded linear operator.
  	The pair $(\{f_n\}_n, S)$ is said to be a Banach frame for $\mathcal{X}$ w.r.t $\mathcal{X}_d$ 
  	if the following holds.
  	\begin{enumerate}[\upshape(i)]
  		\item $\{f_n(x)\}_n \in \mathcal{X}_d$, for  each $x \in \mathcal{X}$. 
  		\item There exist $a,b>0$ such that 
  		$
  			a\|x\|\leq \|\{f_n(x)\}_n\|\leq b\|x\|,  \forall x \in \mathcal{X}.
  	$
  		\item $S(\{f_n(x)\}_n)=x$, for  each $x \in \mathcal{X}$.
  	\end{enumerate}
  	Constants $a$ and $b$ are called as Banach frame bounds. The operator $S$ is called as reconstruction operator.
  \end{definition}
  We again look at Theorem \ref{FUNDAMENTALTHEOREM}. We now define $f_n (h)\coloneqq \langle h,  \tau_n\rangle $, $ \forall h \in \mathcal{H}$, $ \forall n$ and $S\coloneqq S_\tau^{-1}\theta_\tau^*$. Then $(\{f_n\}_n, S)$ is  a Banach frame for $\mathcal{H}$ w.r.t $\ell^2(\mathbb{N})$.\\
  After the formation of Definitions  \ref{ATOMICDECOMPODEFI} and \ref{BANACHFRAMEDEF}, Feichtinger and Grochenig in 90's 
  \cite{FEICHTINGGERGROCHENIG1, FEICHTINGGERGROCHENIG2, FEICHTINGGERGROCHENIG3},   developed
  a  theory of atomic decompositions and frames for a large class of function spaces (modulation spaces \cite{FEICHTINGERLOOKING}  and coorbit  spaces), via,
  group representations and projective representations.  Christensen and Heil \cite{CHRISTENSENHEIL} studied the stability of atomic
  decompositions and Banach frames. For more on atomic decompositions we refer \cite{CHRISTENSENATOMICVIA, CARANDOLASSALLE} and for more on Banach frames we refer 
  \cite{GROCHENIGLOCALIZATION, ALDROUBISLANTED, GITTASHEARLET,  CARANDOLASSALLESCHMIDBERG, TEREKHIN2010, TEREKHIN2009, TEREKHIN2004, FORNASIERALPHA, FORNASIERGROIN}.\\
  Abstract study of atomic decompositions and Banach frames started from the
  fundamental paper \cite{CASAZZAHANLARSONFRAMEBANACH} of Casazza, Han, and Larson in 1999. As a sample, consider
  the abstract question: Does every separable Banach space admit a Banach frame?
  For Hilbert spaces it follows immediately, since separable spaces have
  orthonormal bases and an orthonormal basis is a frame. Existence of Banach frame
  can not be argued by considering the  existence of Schauder basis, since there
  are Banach spaces without Schauder bases. Using Hahn-Banach theorem, the following result was proved in \cite{CASAZZAHANLARSONFRAMEBANACH}.
  \begin{theorem}\cite{CASAZZAHANLARSONFRAMEBANACH}\label{BANACHFRAMEEXISTSSEPARABLE}
  	Every separable Banach space admits a Banach frame w.r.t. a closed subspace of $\ell^\infty(\mathbb{N}).$
  \end{theorem}
  We now note that a direct version of Definition \ref{FREMEDEF} had been set for Banach spaces
  by Aldroubi in 2001 \cite{ALDROUBISUNTANG}.
  \begin{definition}\label{ALDROUBITANG}\cite{ALDROUBISUNTANG}
  	Let $1<p<\infty$ and $\mathcal{X}$ be a Banach space. A sequence 
  	$\{f_n\}_{n}$ of bounded linear functionals in $\mathcal{X}^*$ is said to be a
  	p-frame for $\mathcal{X}$ if there exist $a,b>0$ such that 	
  	\begin{align*}
  		a\|x\|\leq \left(\sum_{n=1}^{\infty}|f_n(x)|^p\right)^\frac{1}{p}\leq
  		b\|x\|,\quad \forall x \in \mathcal{X}.
  	\end{align*}
  \end{definition}

  Like frames for Hilbert spaces,  p-frames for Banach spaces can be characterized using operators. 
  \begin{theorem}\cite{CHRISTENSENSTOEVA}\label{pFRAMECHAR}
  	Let $\mathcal{X}$ be  a Banach space and $\{f_n\}_{n}$ be a sequence in
  	$\mathcal{X}^*$.
  	\begin{enumerate}[\upshape(i)]
  		\item $\{f_n\}_{n}$ is a p-Bessel sequence for 
  		$\mathcal{X}$ with bound $b$ if and only if 
  		\begin{align}\label{BASSELOPERATORCHARACTERIZATION}
  			T: \ell^q (\mathbb{N}) \ni \{a_n\}_{n} \mapsto  \sum_{n=1}^\infty a_nf_n \in \mathcal{X}^*
  		\end{align}
  		is a well-defined (hence bounded) linear operator and $\|T\|\leq b$ (where $q$ is the conjugate
  		index of $p$).
  		\item If $\mathcal{X}$ is reflexive, then $\{f_n\}_{n}$ is a p-frame 
  		for $\mathcal{X}$ if and only if the operator $T$ in
  		(\ref{BASSELOPERATORCHARACTERIZATION}) is surjective.
  	\end{enumerate}
  \end{theorem}
  For further results on p-frames we refer \cite{LIUFEIC, STOEVALOWER, STOEVA2}.
  There are two generalizations of Definition \ref{ALDROUBITANG} which we state now.
  \begin{definition}\cite{CASAZZACHRISTENSENSTOEVA}\label{XDFRAME}
  	Let $\mathcal{X}$ be a Banach space and $\mathcal{X}_d$  be  an associated  BK-space. A collection $\{f_n\}_n$  in  $\mathcal{X}^*$ is said to be a  $\mathcal{X}_d$-frame 
  	for $\mathcal{X}$ if the following holds.
  	\begin{enumerate}[\upshape(i)]
  		\item $\{f_n(x)\}_n \in \mathcal{X}_d$, for  each $x \in \mathcal{X}$. 
  		\item There exist $a,b>0$ such that $	a\|x\|\leq \|\{f_n(x)\}_n\|\leq b\|x\|,  \forall x \in \mathcal{X}.$
  		\end{enumerate}
  	Constants $a$ and $b$ are called as $\mathcal{X}_d$-frame bounds. 
  \end{definition}
\begin{definition}\cite{CASAZZACHRISTENSENSTOEVA}
		Let $\mathcal{X}$ be a Banach space and $\mathcal{X}_d$  be  an associated  BK-space. A collection $\{\tau_n\}_n$ in $\mathcal{X}$ is said to be a  $\mathcal{X}_d$-frame 
	for $\mathcal{X}$ if the following holds.
		\begin{enumerate}[\upshape(i)]
		\item $\{f(\tau_n)\}_n \in \mathcal{X}_d$, for  each $f \in \mathcal{X}^*$. 
		\item There exist $a,b>0$ such that $	a\|f\|\leq \|\{f(\tau_n)\}_n\|\leq b\|f\|,  \forall f \in \mathcal{X}^*.$
		\end{enumerate}
	 \end{definition}
  Note that Definition \ref{XDFRAME} is the  Definition \ref{BANACHFRAMEDEF} of Banach frame without the third condition.

It is known that an $\mathcal{X}_d$-frame for a Banach space need not admit representation of every element of the Banach space. With regard to this, the following result gives information when it is possible to express element of the Banach space using series. 
  \begin{theorem}\cite{CASAZZACHRISTENSENSTOEVA}\label{CASAZZASEPARABLECHARACTERIZATION}
  	Let   $\mathcal{X}_d$ be a BK-space and 
  	$\{f_n\}_n$ be an $\mathcal{X}_d$-frame for $\mathcal{X}$. Let $ \theta_f:\mathcal{X} \ni x \mapsto \{f_n(x)\}_n \in \mathcal{X}_d$ (this map is a well-defined  linear bounded below operator). 
  	The following are equivalent. 
  	\begin{enumerate}[\upshape(i)]
  		\item $ \theta_f(\mathcal{X})$ is complemented in $\mathcal{X}_d$.
  		\item The operator $\theta_f^{-1}:\theta_f(\mathcal{X}) \rightarrow \mathcal{X}$ can be extended to a bounded linear operator $T_f: \mathcal{X}_d \rightarrow \mathcal{X}.$
  		\item There exists a bounded linear operator $S: \mathcal{X}_d \rightarrow \mathcal{X}$ such that $(\{f_n\}_n, S)$ is a Banach frame for  $\mathcal{X}$ w.r.t  $\mathcal{X}_d$.
  	\end{enumerate}
  	Also, the condition 
  	\begin{enumerate}[\upshape(i)]\addtocounter{enumi}{3}
  		\item  There exists a sequence $\{\tau_n\}_n$ in $\mathcal{X}$ such that $\sum_{n=1}^\infty a_n \tau_n$ is convergent in $\mathcal{X}$
  		for all $\{a_n\}_n$  in $\mathcal{X}_d$ and  $x=\sum_{n=1}^\infty f_n(x) \tau_n, \forall x \in \mathcal{X}.$
  	\end{enumerate}
  	implies each of (i)-(iii). If we also assume that the canonical unit vectors $\{e_n\}_n$ form a Schauder basis for $\mathcal{X}_d$,
  	(iv) is equivalent to the above (i)-(iii) and to the following condition (v).
  	\begin{enumerate}[\upshape(i)]\addtocounter{enumi}{4}
  		\item There exists an $\mathcal{X}_d^*$-Bessel sequence  $\{\tau_n\}_n\subseteq \mathcal{X}\subseteq \mathcal{X}^{**}$ for 
  		$\mathcal{X}^*$ such that $x=\sum_{n=1}^\infty f_n(x) \tau_n, $ $ \forall x \in \mathcal{X}.$
  	\end{enumerate}
  	If the canonical unit vectors $\{e_n\}_n$ form a Schauder basis for both $\mathcal{X}_d$ and $\mathcal{X}_d^*$, then (i)-(v) is equivalent
  	to the following condition (vi).
  	\begin{enumerate}[\upshape(i)]\addtocounter{enumi}{5}
  		\item There exists an $\mathcal{X}_d^*$-Bessel sequence  $\{\tau_n\}_n\subseteq \mathcal{X}\subseteq \mathcal{X}^{**}$ for 
  		$\mathcal{X}^*$ such that $f=\sum_{n=1}^\infty f(\tau_n) f_n, $ $ \forall f \in \mathcal{X}^*.$
  		\end{enumerate}
  	In each of the cases (v) and (vi), $\{\tau_n\}_n$ is actually an $\mathcal{X}_d^*$-frame for $\mathcal{X}^*$. 
  	\end{theorem}
  For further results concerning $\mathcal{X}_d$-frames, we refer the reader to \cite{STOEVAGEN, STOEVA2012, STOEVA2009}.\\
 
  In this paper we make a generalization of frame theory from Banach spaces to metric spaces. We organized this paper as follows. We first define the notion of frames for metric spaces. We give several examples and derive that separable metric spaces admit a metric $\mathcal{M}_d$-frame (Theorem \ref{METRICFRAMEEXISTS3}). We obtain some results on metric frames and then define the notion of metric frame in analogy with the notion of Banach frames (Definition \ref{METRICBANACHFRAME}). Here we give examples and provide certain partial results about existence (Theorems \ref{METRICFRAMEEXISTS} and \ref{METRICFRAMEEXISTS2}). Then we turn into representation of elements in Banach spaces using metric frames. We obtain results, in certain cases,  analogous to Theorem \ref{pFRAMECHAR} and Theorem \ref{CASAZZASEPARABLECHARACTERIZATION} (Theorems \ref{PBESSELCHAR} and \ref{PALL}, respectively). Using Arens-Eells theorem we show that we can always switch from metric space to Banach space (Theorem \ref{LIPIFFLINEAR}). We end the article by deriving some stability results (Theorems \ref{FIRSTPERTURB} and  \ref{STABILITYMA}). 
   
  \section{Frames for metric spaces}\label{FRAMESFORMETRICSPACES}
  To define frames for metric spaces we need some notions. Let $\mathcal{M}$, 
  $\mathcal{N}$ be  metric spaces. We recall that  a function $f:\mathcal{M}  \rightarrow
  \mathcal{N}$ is said to be Lipschitz if there exists $b> 0$ such that 
  \begin{align*}
  	d(f(x), f(y)) \leq b\, d(x,y), \quad \forall x, y \in \mathcal{M}.
  \end{align*}
  A Lipschitz function $f:\mathcal{M}  \rightarrow
  \mathcal{N}$ is said to be bi-Lipschitz if there exists $a> 0$ such that 
  \begin{align*}
  	a\, d(x,y) \leq d(f(x), f(y)) , \quad \forall x, y \in \mathcal{M}.
  \end{align*}
Whenever codomain of Lipschitz function is a Banach space, we have structure and distance on the space of all Lipschitz functions. These are obtained in the next definition and theorem.
  \begin{definition}\cite{WEAVER}
  	Let	$\mathcal{X}$ be a Banach space.
  	\begin{enumerate}[\upshape(i)]
  		\item Let $\mathcal{M}$ be a  metric space. The collection 	$\operatorname{Lip}(\mathcal{M}, \mathcal{X})$
  		is defined as $\operatorname{Lip}(\mathcal{M}, \mathcal{X})\coloneqq \{f:\mathcal{M} 
  		\rightarrow \mathcal{X}  \operatorname{ is ~ Lipschitz} \}.$ For $f \in \operatorname{Lip}(\mathcal{M}, \mathcal{X})$, the Lipschitz number 
  		is defined as 
  		\begin{align*}
  			\operatorname{Lip}(f)\coloneqq \sup_{x, y \in \mathcal{M}, x\neq
  				y} \frac{\|f(x)-f(y)\|}{d(x,y)}.
  		\end{align*}
  		\item Let $(\mathcal{M}, 0)$ be a pointed metric space. The collection 	$\operatorname{Lip}_0(\mathcal{M}, \mathcal{X})$
  		is defined as $\operatorname{Lip}_0(\mathcal{M}, \mathcal{X})\coloneqq \{f:\mathcal{M} 
  		\rightarrow \mathcal{X}  \operatorname{ is ~ Lipschitz ~ and } f(0)=0\}.$
  		For $f \in \operatorname{Lip}_0(\mathcal{M}, \mathcal{X})$, the Lipschitz norm
  		is defined as 
  		\begin{align*}
  			\|f\|_{\operatorname{Lip}_0}\coloneqq \sup_{x, y \in \mathcal{M}, x\neq
  				y} \frac{\|f(x)-f(y)\|}{d(x,y)}.
  		\end{align*}
  	\end{enumerate}
  	
  \end{definition}
  
  \begin{theorem}\cite{WEAVER}\label{BANACHALGEBRA}
  	Let	$\mathcal{X}$ be a Banach space.
  	\begin{enumerate}[\upshape(i)]
  		\item If $\mathcal{M}$ is a  metric space, then 	$\operatorname{Lip}(\mathcal{M},
  		\mathcal{X})$ is a semi-normed vector  space w.r.t. the semi-norm  $\operatorname{Lip}(\cdot)$.
  		\item If $(\mathcal{M}, 0)$ is a pointed metric space, then 	$\operatorname{Lip}_0(\mathcal{M},
  		\mathcal{X})$ is a Banach space w.r.t. the  norm
  		$\|\cdot\|_{\operatorname{Lip}_0}$. Further, $\operatorname{Lip}_0(\mathcal{X})\coloneqq\operatorname{Lip}_0(\mathcal{X},
  		\mathcal{X})$ is a unital Banach algebra. In particular, if $T \in \operatorname{Lip}_0(\mathcal{X})$ satisfies $
  			\|T-I_\mathcal{X}\|_{\operatorname{Lip}_0}<1,$ 
  		then $T $ is invertible and $T^{-1} \in \operatorname{Lip}_0(\mathcal{X})$.
  	\end{enumerate}
  \end{theorem}

  With these information we can now define frames for metric spaces.
  \begin{definition}\label{FIRST}(p-frame for metric space)
  	Let $(\mathcal{M},d)$,  $(\mathcal{N}_n,d_n)$, $n \in \mathbb{N}$ be  
  	metric
  	spaces and $p \in (0,\infty)$. A collection $\{f_n\}_{n}$ of Lipschitz functions,  $f_n\in \operatorname{Lip}(\mathcal{M},
  	 \mathcal{N}_n)$  is said to be a metric p-frame or Lipschitz p-frame  for  $\mathcal{M}$ relative to
  	$\{\mathcal{N}_n\}_n$ if 
  	there exist $a,b>0$ such that 
  	\begin{align*}
  	a\,d(x,y)\leq \left(\sum_{n=1}^{\infty}d_n(f_n(x),
  	f_n(y))^p\right)^\frac{1}{p}\leq b\,d(x,y),\quad \forall x, y \in \mathcal{M}.
  	\end{align*}
  	If $a$ is allowed to take the value 0, then we say that $\{f_n\}_{n}$  a  metric p-Bessel sequence for  $\mathcal{M}$.
  \end{definition}
Throughout the paper, we study  the following  particular case of Definition \ref{FIRST}.
  \begin{definition}\label{PFRAMEFORMETRIC}(p-frame for metric space w.r.t. scalars)
  	Let $\mathcal{M}$ be a metric space and $p \in (0,\infty)$. A collection $\{f_n\}_{n}$ of Lipschitz functions,   $f_n \in \operatorname{Lip}(\mathcal{M}, \mathbb{K})$  is said to be a metric p-frame or Lipschitz p-frame for  $\mathcal{M}$ if there exist $a,b>0$ such that 
  	\begin{align*}
  		a\,d(x,y)\leq \left(\sum_{n=1}^{\infty}|f_n(x)-f_n(y)|^p\right)^\frac{1}{p}\leq b\,d(x,y),\quad \forall x, y \in \mathcal{M}.
  	\end{align*}
  	If we do not demand the first inequality, then we say $\{f_n\}_{n}$ is metric p-Bessel sequence for  $\mathcal{M}$.
  \end{definition}
  We now see that whenever $\mathcal{M}$ is a Banach space and $f_n$'s are linear functionals, then Definition \ref{PFRAMEFORMETRIC} reduces to Definition \ref{ALDROUBITANG}.  Before proceeding, we give  various examples.
  \begin{example}
  	Let $\{f_n\}_{n}$ be a p-frame for a Banach space $\mathcal{X}$. Choose any bi-Lipschitz function $A:\mathcal{X}\to \mathcal{X}$. Then it follows that $\{f_nA\}_{n}$ is a metric  p-frame for $\mathcal{X}$.
  \end{example}
  \begin{example}\label{1FRAMEFIRST}
  	Let $1<a<\infty.$ Let us take $\mathcal{M}\coloneqq[a,\infty)$ and define $f_n:\mathcal{M}\to \mathbb{R}$ by 
  	\begin{align*}
  		f_0(x)&\coloneqq 1, \quad \forall x \in \mathcal{M}\\
  		f_n(x)&\coloneqq \frac{(\log x)^n}{n!}, \quad \forall x \in \mathcal{M}, \forall n\geq 1.
  	\end{align*}
  	Then $f_n'(x)=\frac{(\log x)^{(n-1)}}{(n-1)!x}$, $\forall x \in \mathcal{M}, \forall n\geq1.$ Since   $f_n'$ is bounded on $\mathcal{M}$, $\forall n\geq1$ it follows that $f_n$ is Lipschitz on $\mathcal{M}$, $\forall n\geq1$. For $x, y \in \mathcal{M},$ with $x<y$, we now see that 
  	\begin{align*}
  		\sum_{n=0}^{\infty}|f_n(x)-f_n(y)|=\sum_{n=0}^{\infty}\frac{(\log
  			y)^n}{n!}-\sum_{n=0}^{\infty}\frac{(\log x)^n}{n!}
  		=e^{\log y}-e^{\log x}=y-x=|x-y|.
  	\end{align*}
  	Hence $\{f_n\}_n$ is a metric 1-frame for $\mathcal{M}$.
  \end{example}
  \begin{example}\label{1FRAMESECOND}
  	Let $1<a<b<\infty.$ We  take $\mathcal{M}\coloneqq[\frac{1}{1-a},\frac{1}{1-b}]$ and define $f_n:\mathcal{M}\to \mathbb{R}$ by 
  	\begin{align*}
  		f_n(x)&\coloneqq\left(1-\frac{1}{x}\right)^n, \quad \forall x \in \mathcal{M}, \forall n \geq 0.
  	\end{align*}
  	Then $f_n'(x)=\frac{n}{-x^2}\left(1-\frac{1}{x}\right)^{n-1}$, $\forall x \in \mathcal{M}, \forall n\geq1.$ Therefore $f_n$ is a Lipschitz function, for each $n\geq1.$ We now see that $\{f_n\}_n$ is a metric 1-frame for $\mathcal{M}$. In fact, for  $x, y \in \mathcal{M},$ with $x<y$, we have
  	\begin{align*}
  		\sum_{n=0}^{\infty}|f_n(x)-f_n(y)|=\sum_{n=0}^{\infty}\left(1-\frac{1}{y}\right)^n-\sum_{n=0}^{\infty}\left(1-\frac{1}{x}\right)^n
  		=y-x=|x-y|.
  	\end{align*}
  	\end{example}
  \begin{example}\label{LINEARGOOD}
  	Let $\{f_n\}_{n}$ be a p-frame for a Banach space $\mathcal{X}$. Let $\phi: \mathbb{K}\to  \mathbb{K}$ be bi-Lipschitz  and define   $g_n (x)\coloneqq \phi (f_n(x)), $ $\forall x \in \mathcal{X}, $  $\forall n \in \mathbb{N}$. It then follows that $\{g_n\}_n  $ is a metric p-frame for $\mathcal{X}$.
  \end{example}

  By looking at Theorem \ref{pFRAMECHAR} we can ask whether there is a result similar for metric p-frames. We answer this partially through the following theorem.
  \begin{theorem}\label{PBESSELCHAR}
  	Let $(\mathcal{M},0)$ be  a pointed metric  space and $\{f_n\}_{n}$ be a sequence in $\operatorname{Lip}_0(\mathcal{M},
  	\mathbb{K})$.  Then $\{f_n\}_{n}$ is a metric p-Bessel sequence for 
  	$\mathcal{M}$ with bound $b$ if and only if 
  	\begin{align}\label{LIPBASSELOPERATORCHARACTERIZATION}
  		&T: \ell^q (\mathbb{N})\ni \{a_n\}_{n} \mapsto T\{a_n\}_{n} \in \operatorname{Lip}_0(\mathcal{M}\times \mathcal{M},
  		\mathbb{K}),\\
  		&T\{a_n\}_{n}: \mathcal{M}\times \mathcal{M} \ni (x,y) \mapsto  \sum_{n=1}^\infty a_n(f_n(x)-f_n(y)) \in \mathbb{K} \nonumber
  	\end{align}
  	is a well-defined (hence bounded) operator and $\|T\|\leq b$ (where $q$ is conjugate
  	index of $p$).
  	
  \end{theorem}
  \begin{proof}
  	$(\Rightarrow)$ Let $\{a_n\}_{n} \in \ell^q (\mathbb{N})$ and $n, m\in \mathbb{N}$ with $n<m$.  First we have to show that the series in (\ref{LIPBASSELOPERATORCHARACTERIZATION}) is convergent. For all $x, y \in \mathcal{M}$, 
  	\begin{align*}
  		\left|\sum_{k=n}^{m}a_k(f_k(x)-f_k(y))\right|&\leq \left(\sum_{k=n}^{m}|a_k|^q\right)^\frac{1}{q}\left(\sum_{k=n}^{m}|f_k(x)-f_k(y)|^p\right)^\frac{1}{p}\\
  		&\leq b \left(\sum_{k=n}^{m}|a_k|^q\right)^\frac{1}{q}\, d(x,y).
  	\end{align*}
  	Therefore the series in (\ref{LIPBASSELOPERATORCHARACTERIZATION})  converges. We next show that the map $T\{a_n\}_{n}$ is Lipschitz. Consider 
  		\begin{align*}
  		\left\|T\{a_n\}_{n}\right\|_{\operatorname{Lip}_0}& =\sup_{(x, y), (u,v) \in \mathcal{M}\times \mathcal{M}, (x, y)\neq (u,v)}\frac{|T\{a_n\}_{n}(x,y)-T\{a_n\}_{n}(u,v)|}{d(x,u)+d(y,v)}\\
  		&=\sup_{(x, y), (u,v) \in \mathcal{M}\times \mathcal{M}, (x, y)\neq (u,v)}\frac{|\sum_{n=1}^{\infty}a_n(f_n(x)-f_n(u))-\sum_{n=1}^{\infty}a_n(f_n(y)-f_n(v))|}{d(x,u)+d(y,v)}\\
  		&\leq \sup_{(x, y), (u,v) \in \mathcal{M}\times \mathcal{M}, (x, y)\neq (u,v)}\frac{|\sum_{n=1}^{\infty}a_n(f_n(x)-f_n(u))|+|\sum_{n=1}^{\infty}a_n(f_n(y)-f_n(v))|}{d(x,u)+d(y,v)}\\
  		&\leq b\sup_{(x, y), (u,v) \in \mathcal{M}\times \mathcal{M}, (x, y)\neq (u,v)}\frac{\left(\sum_{n=1}^{\infty}|a_n|^q\right)^\frac{1}{q}\, d(x,u)+\left(\sum_{n=1}^{\infty}|a_n|^q\right)^\frac{1}{q}\, d(y,v)}{d(x,u)+d(y,v)}\\
  		&=b\left(\sum_{n=1}^{\infty}|a_n|^q\right)^\frac{1}{q}.
  	\end{align*}
  	Hence $T$ is well-defined. Clearly $T$ is linear.  Boundedness of $T$ with bound $b$ will follow from previous calculation.\\
  	$(\Leftarrow)$	Banach-Steinhaus theorem tells that $T$ is bounded. Given $x, y \in \mathcal{M}$,  we define a map 
  	\begin{align*}
  		\Phi_{x,y}: \ell^q (\mathbb{N})  \ni \{a_n\}_{n} \mapsto \Phi_{x,y}\{a_n\}_{n}\coloneqq \sum_{n=1}^{\infty}a_n(f_n(x)-f_n(y))\in \mathbb{K}
  	\end{align*} 
  	which is a bounded linear functional. Hence $\{f_n(x)-f_n(y)\}_{n}\in \ell^p (\mathbb{N})$. Let $\{e_n\}_{n}$ be the standard Schauder basis for $ \ell^p (\mathbb{N})$. Then 
  	\begin{align*}
  		\|\Phi_{x,y}\|=\left(\sum_{n=1}^{\infty}|\Phi_{x,y}\{e_n\}_{n}|^p\right)^\frac{1}{p}= \left(\sum_{n=1}^{\infty}|f_n(x)-f_n(y)|^p\right)^\frac{1}{p}.
  	\end{align*}
  	Now 
  	\begin{align*}
  		b\left(\sum_{n=1}^{\infty}|a_n|^q\right)^\frac{1}{q}&=b\|\{a_n\}_{n}\|\geq \|T\{a_n\}_{n}\|_{\operatorname{Lip}_0}\\
  		&\geq \sup_{(x, 0), (y,0) \in \mathcal{M}\times \mathcal{M}, (x, 0)\neq (y,0)}\frac{|T\{a_n\}_{n}(x,0)-T\{a_n\}_{n}(y,0)|}{d(x,y)}\\
  		&=\sup_{(x, 0), (y,0) \in \mathcal{M}\times \mathcal{M}, (x, 0)\neq (y,0)}\frac{|\sum_{n=1}^{\infty}a_n(f_n(x)-f_n(y))|}{d(x,y)}\\
  		&=\sup_{(x, 0), (y,0) \in \mathcal{M}\times \mathcal{M}, (x, 0)\neq (y,0)}\frac{|\Phi_{x,y}\{a_n\}_{n}|}{d(x,y)}
  	\end{align*}
  	which implies 
  	\begin{align*}
  		|\Phi_{x,y}\{a_n\}_{n}|\leq b \left(\sum_{n=1}^{\infty}|a_n|^q\right)^\frac{1}{q}\,d(x,y), \quad \forall x, y \in \mathcal{M}.
  	\end{align*}
  	Using all these, we finally get 
  	\begin{align*}
  		\left(\sum_{n=1}^{\infty}|f_n(x)-f_n(y)|^p\right)^\frac{1}{p}=\|\Phi_{x,y}\|\leq b\, d(x,y), \quad \forall x, y \in \mathcal{M}.
  	\end{align*}
  Hence 	$\{f_n\}_{n}$ is a metric p-Bessel sequence for 
  	$\mathcal{M}$ with bound $b$.
  \end{proof}
  Previous theorem leads to the question: What are all the metric spaces for which the following statement holds? 
 ``$\{f_n\}_{n}$ is a metric p-frame 
  for $\mathcal{M}$ if and only if the operator $T$ in
  (\ref{LIPBASSELOPERATORCHARACTERIZATION}) is surjective".\\
  In the spirit of definition of $\mathcal{X}_d$-frame, Definition \ref{PFRAMEFORMETRIC} can be generalized.
  
  \begin{definition}\label{XDMETRICFRAME}
  	Let $\mathcal{M}$ be a metric space and $\mathcal{M}_d$ be an associated BK-space. Let
  	$\{f_n\}_{n}$ be a sequence in $\operatorname{Lip}(\mathcal{M}, \mathbb{K})$. If: 
  	\begin{enumerate}[\upshape(i)]
  		\item $\{f_n(x)\}_{n} \in \mathcal{M}_d$, for each  $x \in \mathcal{M}$,
  		\item There exist positive $a, b$  such that 
  		$
  			a\, d(x,y) \leq \|\{f_n(x)-f_n(y)\}_n\|_{\mathcal{M}_d} \leq b\, d(x,y),  \forall x
  			, y\in \mathcal{M},
  		$
  	\end{enumerate}
  	then we say that $\{f_n\}_{n}$ is a  metric $\mathcal{M}_d$-frame (or Lipschitz $\mathcal{M}_d$-frame) for $\mathcal{M}$ w.r.t. $\mathcal{M}_d$. 	If we do not demand the first inequality, then we say $\{f_n\}_{n}$ is metric $\mathcal{M}_d$-Bessel sequence.
  \end{definition}
  An easier  way of producing metric $\mathcal{M}_d$-frame is the following. Let $\mathcal{M}_d$ be a BK-space which admits a Schauder basis $\{\tau_n\}_{n}$. Let  $\{f_n\}_{n}$ be the coefficient functionals associated with $\{\tau_n\}_{n}$. Let $\mathcal{M}$ be a metric space and $A:\mathcal{M} \rightarrow \mathcal{M}_d$ be bi-Lipschitz with bounds $a$ and $b$. Define $g_n\coloneqq f_n A, \forall n$. Then $g_n$ is a Lipschitz function for all $n$. Now 
  \begin{align*}
  	a\, d(x,y) &\leq \|Ax-Ay\|_{\mathcal{M}_d}=\|\{f_n(Ax-Ay)\}_n\|_{\mathcal{M}_d}\\
  	&=\|\{f_n(Ax)-f_n(Ay)\}_n\|_{\mathcal{M}_d} 
  	=\|\{g_n(x)-g_n(y)\}_n\|_{\mathcal{M}_d} \leq b\, d(x,y), \quad \forall x
  	, y\in \mathcal{M}.
  \end{align*}
  Hence $\{g_n\}_{n}$ is a  metric $\mathcal{M}_d$-frame for $\mathcal{M}$ w.r.t. $\mathcal{M}_d$.\\
  Following result ensures that metric frames are universal in nature.
  \begin{theorem}\label{METRICFRAMEEXISTS3}
  	Every separable metric space $ \mathcal{M}$ admits a metric $\mathcal{M}_d$-frame.
  \end{theorem}
  \begin{proof}
  	From Aharoni's theorem \cite{KALTONLANCIEN} it follows that there exists a bi-Lipschitz function $f: \mathcal{M} \to c_0(\mathbb{N})$. Let $ p_n: c_0(\mathbb{N}) \to \mathbb{K}$ be the coordinate projection, for each $n$. If we now set 	$f_n\coloneqq p_nf$, for each $n$, then $\{f_n\}_{n}$ is a metric frame  for $\mathcal{M}$ w.r.t. $c_0(\mathbb{N})$.
  \end{proof}
Given metric $\mathcal{M}_d$-frames  $\{f_n\}_{n}$, $\{g_n\}_{n}$ and a nonzero  scalar $\lambda$, one naturally asks whether we can scale and add them to get new frames? i.e.,  whether $\{f_n+\lambda g_n\}_{n}$ is a frame? In the case of Hilbert spaces, a use of Minkowski's inequality shows that whenever $\{\tau_n\}_{n}$ and $\{\omega_n\}_{n}$ are frames for a Hilbert space $\mathcal{H}$, then the linear combination $\{\tau_n+\lambda \omega_n\}_{n}$ is a Bessel sequence for $\mathcal{H}$. In general this sequence need not be a frame for $\mathcal{H}$. Thus we have to impose extra conditions to ensure the existence of  lower frame bound. For Hilbert spaces this is done by Favier and Zalik \cite{FAVIERZALIK}. We now obtain similar results for metric spaces.
\begin{theorem}
Let 	$\{f_n\}_{n}$ be a  metric $\mathcal{M}_d$-frame for metric space $\mathcal{M}$ with bounds $a$ and $b$. Let $\lambda$ be a non-zero  scalar. Then 
\begin{enumerate}[\upshape(i)]
	\item   $\{\lambda f_n\}_{n}$ is a metric  $\mathcal{M}_d$-frame for  $\mathcal{M}$ with bounds $a\lambda$ and $b\lambda$.
	\item If $A:\mathcal{M} \rightarrow \mathcal{M}$ is bi-Lipschitz with bounds $c$ and $d$, then $\{f_nA\}_{n}$ is a metric  $\mathcal{M}_d$-frame for  $\mathcal{M}$ with bounds $ac$ and $bd$.
	\item If $\{g_n\}_{n}$ is a metric $\mathcal{M}_d$-Bessel sequence for $\mathcal{M}$ with bound $d$ and $|\lambda|<\frac{a}{d}$, then $\{f_n+\lambda g_n\}_{n}$ is a metric $\mathcal{M}_d$-frame for  $\mathcal{M}$ with bounds $a-|\lambda|d$ and $b+|\lambda|d$.
\end{enumerate}
\end{theorem}
\begin{proof}
First two conclusions follow easily. For the upper frame bound of $\{f_n+\lambda g_n\}_{n}$ 	we use  triangle inequality. Now for lower frame bound, using reverse triangle inequality, we get 

\begin{align*}
\|\{(f_n+\lambda g_n)(x)-(f_n+\lambda g_n)(y)\}_n\|_{\mathcal{M}_d}&=\|\{f_n(x)-f_n(y)+ \lambda( g_n(x)- g_n(y))\}_n\|_{\mathcal{M}_d}\\
&\geq \|\{f_n(x)-f_n(y)\}_n\|_{\mathcal{M}_d}-\|\{ \lambda( g_n(x)- g_n(y))\}_n\|_{\mathcal{M}_d}\\
 &\geq  a\, d(x,y)- |\lambda| \, d(x,y)
 =(a-|\lambda|)\, d(x,y), \quad \forall x
, y\in \mathcal{M}.
\end{align*}
\end{proof}
We next define  ``metric frame" which is stronger than    Definition \ref{XDMETRICFRAME} in light of definition of Banach frame.
  
  \begin{definition}\label{METRICBANACHFRAME}
  	Let $\mathcal{M}$ be a metric space and $\mathcal{M}_d$ be an associated  BK-space. Let
  	$\{f_n\}_{n}$ be a sequence in $\operatorname{Lip}(\mathcal{M}, \mathbb{K})$
  	and $S: \mathcal{M}_d \rightarrow \mathcal{M}$. If: 
  	\begin{enumerate}[\upshape(i)]
  		\item $\{f_n(x)\}_{n} \in \mathcal{M}_d$, for each  $x \in \mathcal{M}$,
  		\item There exist positive $a, b$  such that 
  		$
  			a\, d(x,y) \leq \|\{f_n(x)-f_n(y)\}_n\|_{\mathcal{M}_d} \leq b\, d(x,y),  \forall x
  			, y\in \mathcal{M},
  		$
  		\item $S$ is Lipschitz and $S(\{f_n(x)\}_{n})=x$, for each $x \in \mathcal{M}$,
  	\end{enumerate}
  	then we say that $(\{f_n\}_{n}, S)$ is a metric frame or Lipschitz metric
  	frame for $\mathcal{M}$ w.r.t. $\mathcal{M}_d$. Mapping $S$ is called as
  	Lipschitz reconstruction operator. We say  constant $a$ as lower frame bound
  	and constant $b$ as upper frame bound. If we do not demand the first inequality, then we say $(\{f_n\}_{n}, S)$ is a metric  Bessel sequence.
  \end{definition}
  We observe that if $(\{f_n\}_{n}, S)$ is a metric frame  for $\mathcal{M}$ w.r.t. $\mathcal{M}_d$, then condition (i) in Definition \ref{METRICBANACHFRAME} tells 
  that the mapping (we call as analysis map)
  \begin{align*}
  	\theta_f:\mathcal{M} \ni x \mapsto \theta_f x\coloneqq \{f_n(x)\}_{n} \in \mathcal{M}_d
  \end{align*}
  is well-defined and condition (ii) in Definition \ref{METRICBANACHFRAME} tells that $\theta_f$ satisfies 
  \begin{align*}
  	a\, d(x,y)\leq \|\theta_f x -\theta_fy \|\leq b\, d(x,y), \quad \forall x
  	, y\in \mathcal{M}.
  \end{align*}
  Hence $\theta_f$ is bi-Lipschitz and injective. Thus a  metric frame puts the space 
  into $\mathcal{M}_d$ via $\theta_f$ and reconstructs 
  every element by using reconstruction operator $S$. Now note that $S\theta_f =I_\mathcal{M}$. This operator description helps us to derive the following propositions easily.
  \begin{proposition}
  	If $(\{f_n\}_{n}, S)$ is a metric frame  for $\mathcal{M}$ w.r.t. $\mathcal{M}_d$, then 	$P_{f, S}\coloneqq \theta_f S: \mathcal{M}_d \to \mathcal{M}_d$ is idempotent and $P_{f, S}(\mathcal{M}_d)=\theta_f(\mathcal{M}_d).$
  \end{proposition} 
  \begin{proposition}
  	Let $\{f_n\}_{n}$ be a  $\mathcal{M}_d$-frame  for $\mathcal{M}$ w.r.t. $\mathcal{M}_d$ and $S: \mathcal{M}_d \rightarrow \mathcal{M}$ be Lipschitz. Then $(\{f_n\}_{n}, S)$ is a metric frame  for $\mathcal{M}$ w.r.t. $\mathcal{M}_d$ if and only if $S$ is left-Lipschitz inverse of $\theta_f$ if and only if 	$\theta_f$ is right-Lipschitz inverse of $S$.
  \end{proposition}
We now give some explicit examples of metric frames.
  \begin{example}
  	Let $\mathcal{M}$, $\{f_n\}_{n}$ be as in Example \ref{1FRAMEFIRST} and let $a=1$.	Take $\mathcal{M}_d \coloneqq \ell^1(\{0\}\cup \mathbb{N})$ and define 
  	\begin{align*}
  		S:\mathcal{M}_d \ni \{a_n\}_{n} \mapsto S \{a_n\}_{n} \coloneqq 1+\left| \sum_{n=1}^{\infty} a_n\right| \in \mathcal{M}.
  	\end{align*}
  	Then 
  	
  	\begin{align*}
  		|S \{a_n\}_{n}-S \{b_n\}_{n}|&=\left|| \sum_{n=1}^{\infty} a_n|-| \sum_{n=1}^{\infty} b_n| \right|\leq \left| \sum_{n=1}^{\infty} a_n- \sum_{n=1}^{\infty} b_n \right|\\
  		&=\left| \sum_{n=1}^{\infty} (a_n-b_n)\right|\leq \sum_{n=1}^{\infty} |a_n-b_n|\leq \sum_{n=0}^{\infty} |a_n-b_n|\\
  		&=\|\{a_n\}_{n}-\{b_n\}_{n}\|,\quad \forall \{a_n\}_{n}, \{b_n\}_{n} \in \ell^1(\{0\}\cup \mathbb{N}).
  	\end{align*}
  	Thus $S$ is Lipschitz. Further, 
  	\begin{align*}
  		S(\{f_n(x)\}_{n})=1+\left| \sum_{n=1}^{\infty} f_n(x)\right|=1+\sum_{n=1}^{\infty}\frac{(\log x)^n}{n!}=x,\quad \forall x \in \mathcal{M}. 
  	\end{align*}
  	Hence $(\{f_n\}_{n}, S)$ is a metric frame  for $\mathcal{M}$ w.r.t. $\mathcal{M}_d$.
  	Note that if we define \begin{align*}
  		T:\mathcal{M}_d \ni \{a_n\}_{n} \mapsto S \{a_n\}_{n} \coloneqq 1+ \sum_{n=1}^{\infty} |a_n| \in \mathcal{M},
  	\end{align*}
  	then 
  	\begin{align*}
  		|T \{a_n\}_{n}-T \{b_n\}_{n}|&=\left| \sum_{n=1}^{\infty} |a_n|- \sum_{n=1}^{\infty} |b_n| \right|= \left| \sum_{n=1}^{\infty} (|a_n|-|b_n|)  \right|\\
  		&\leq  \sum_{n=1}^{\infty} \bigg||a_n|-|b_n|\bigg|\leq \sum_{n=1}^{\infty} |a_n-b_n|\leq \sum_{n=0}^{\infty} |a_n-b_n|\\
  		&=\|\{a_n\}_{n}-\{b_n\}_{n}\|,\quad \forall \{a_n\}_{n}, \{b_n\}_{n} \in \ell^1(\{0\}\cup \mathbb{N}).
  	\end{align*}
  	Thus $T$ is Lipschitz. Hence  $(\{f_n\}_{n}, T)$ is also  a metric frame  for $\mathcal{M}$ w.r.t. $\mathcal{M}_d$.
  \end{example}

  \begin{example}
  	Let $f_1: \mathbb{K}\to  \mathbb{K}$ be bi-Lipschitz and let  $f_2, \dots, f_m: \mathbb{K}\to  \mathbb{K}$ be Lipschitz maps such that 
  	\begin{align*}
  		f_1(x)+\dots+ f_m(x)=x, \quad \forall x \in \mathbb{K}.
  	\end{align*}
  	We now define $S: \mathbb{K}^m \ni (x_1, \dots, x_m)\mapsto  \sum_{j=1}^{m}x_j \in \mathbb{K}$. Then $(\{f_n\}_{n}, S)$ is   a metric frame  for $\mathbb{K}$ w.r.t. $\mathbb{K}^m$. Note that the operator $S$ is linear. 	
  \end{example}
  After the definition of metric frame, the first question which comes is the
  existence. In Theorem \ref{BANACHFRAMEEXISTSSEPARABLE} it was proved that every separable Banach
  space admits a Banach frame w.r.t. a closed subspace of $\ell^\infty(\mathbb{N})$. Eventhough we don't know this in general, we derive two results one is  close to the definition of metric frame and another gives existence under certain assumptions. To do this
  we want a result which we derive from Mc-Shane extension theorem.
  \begin{theorem}\label{MCSHANE}(Mc-Shane extension theorem)\cite{WEAVER}
  	Let $\mathcal{M}$ be a metric space and $\mathcal{M}_0$  be a nonempty subset
  	of $\mathcal{M}$. If $f_0:\mathcal{M}_0 \rightarrow \mathbb{R} $ is Lipschitz,
  	then there exists a Lipschitz function $f:\mathcal{M} \rightarrow \mathbb{R}
  	$ such that $f|{\mathcal{M}_0}=f_0$ and
  	$\operatorname{Lip}(f)=\operatorname{Lip}(f_0)$.
  \end{theorem}
  Using Theorem \ref{MCSHANE} we derive the following.
  
  \begin{corollary}\label{MCSHANECORO}
  	If $(\mathcal{M}, 0)$ is a pointed metric space, then for every $x
  	\in \mathcal{M}$, there exists a Lipschitz function $f:\mathcal{M} \rightarrow
  	\mathbb{R}$ such that $f(x)=d(x,0)$, $f(0)=0$ and $\operatorname{Lip}(f)=1$.
  \end{corollary}
  \begin{proof}
  	Case (i) : $x\neq0$.
  	Define $\mathcal{M}_0\coloneqq \{0,x\}$ and $f_0(0)=0$, $f_0(x)=d(x,0)$. Then
  	$|f_0(x)-f_0(0)|=d(x,0)$ and hence $f_0$ Lipschitz. Application of Theorem \ref{MCSHANE} now
  	finishes the proof.\\
  	Case (ii) : $x=0$.
  	Take a point $y
  	\in \mathcal{M}$ which is not $0$. We now run the argument in case (i) by
  	putting $y$ in the place of $x$.
  \end{proof}
  \begin{theorem}\label{METRICFRAMEEXISTS}
  	Let $ \mathcal{M}$ be a separable metric space.  Then there exist a BK-space $ \mathcal{M}_d$,  a sequence $\{f_n\}_{n}$  in $\operatorname{Lip}_0(\mathcal{M}, \mathbb{R})$
  	and a function $S: \mathcal{M}_d \rightarrow \mathcal{M}$ such that 
  	\begin{enumerate}[\upshape(i)]
  		\item $\{f_n(x)\}_{n} \in \mathcal{M}_d$, for each  $x \in \mathcal{M}$,
  		\item 
  		$
  		 \|\{f_n(x)-f_n(y)\}_n\|_{\mathcal{M}_d} \leq \, d(x,y), \forall x
  		, y\in \mathcal{M},
  		$
  		\item  $S(\{f_n(x)\}_{n})=x$, for each $x \in \mathcal{M}$.
  	\end{enumerate}
  \end{theorem}
  \begin{proof}
  	Let $\{x_n\}_{n}$ be a dense set in $ \mathcal{M}$.  Then for each $n \in \mathbb{N}$, from Corollary \ref{MCSHANECORO}
  	there exists a Lipschitz function $f_n:\mathcal{M} \rightarrow
  	\mathbb{R}$ such that $f_n(x_n)=d(x_n,0)$, $f_n(0)=0$ and
  	$\operatorname{Lip}(f_n)=1$. Let $ x \in \mathcal{M}$ be fixed. Now for each $n
  	\in\mathbb{N}$, 
  	\begin{align*}
  		|f_n(x)|=|f_n(x)-f_n(0)|\leq \|f_n\|_{\operatorname{Lip}_0}\, d (x,0)=d (x,0)
  	\end{align*}
  	which gives $\sup_{n \in\mathbb{N}}|f_n(x)|\leq d (x,0)$. Since $\{x_n\}_{n}$
  	is dense, there exists a subsequence $\{x_{n_k}\}_{k}$ of $\{x_n\}_{n}$ such
  	that $x_{n_k} \rightarrow x$ as $n \to \infty.$ From the inequality 
  	\begin{align*}
  		|d(y,z)-d(y,w)|\leq d(z,w), \quad \forall y,z,w \in \mathcal{M}
  	\end{align*}
  	we see then that $d(x_{n_k}, 0) \rightarrow d(x,0)$ as $n \to \infty.$ Consider 
  	\begin{align*}
  		d(x_{n_k}, 0)&=f_{n_k}(x_{n_k})\leq 
  		|f_{n_k}(x_{n_k})-f_{n_k}(x)|+|f_{n_k}(x)|\\
  		&\leq 1 .d(x_{n_k},
  		x)+|f_{n_k}(x)|,\quad \forall k \in\mathbb{N}\\
  		\implies &\lim_{k \to \infty}(d(x_{n_k}, 0)-d(x_{n_k},
  		x))\leq\sup_{k \in \mathbb{N} }(d(x_{n_k}, 0)-d(x_{n_k},
  		x)) \leq \sup_{k \in\mathbb{N}}|f_{n_k}(x)|.
  	\end{align*}
  	Therefore 
  	\begin{align*}
  		\sup_{n \in\mathbb{N}}|f_n(x)|&\leq d (x,0)=\lim_{k \to \infty}d(x_{n_k},
  		0)=\lim_{k \to \infty}(d(x_{n_k}, 0)-d(x_{n_k},x))\\
  		&\leq  \sup_{k
  			\in\mathbb{N}}|f_{n_k}(x)|\leq\sup_{n \in\mathbb{N}}|f_n(x)|.
  	\end{align*}
  	So we proved that 
  	\begin{align}\label{ALMOST}
  		d(x,0)=\sup_{n \in\mathbb{N}}|f_n(x)|, \quad \forall x \in \mathcal{M}.
  	\end{align}
  	Define 
  $
  		\mathcal{M}^0_d\coloneqq \{\{f_n(x)\}_n: x \in \mathcal{M}\}.
  	$
  	Equality (\ref{ALMOST}) then tells that $\mathcal{M}^0_d$ is a subset of $\ell^\infty(\mathbb{N}).$ Now we define
  	$S_0:\mathcal{M}_d^0 \ni \{f_n(x)\}_n \mapsto x \in \mathcal{M}$. Then from
  	Equality (\ref{ALMOST}),
  	\begin{align*}
  	d(S_0(\{f_n(x)\}_n),S_0(\{f_n(y)\}_n)&=d(x,y)
  	\leq d(x,0)+d(0,y)\\
  	&=\sup_{n \in\mathbb{N}}|f_n(x)|+\sup_{n \in\mathbb{N}}|f_n(y)|\\
  	&=\|\{f_n(x)\}_n\|+\|\{f_n(y)\}_n\|, \quad \forall x, y \in
  	\mathcal{M}.
  	\end{align*}
  	We will also have 
  	\begin{align*}
  	\|\{f_n(x)-f_n(y)\}_n\|_{\mathcal{M}_d}&=\sup_{n \in\mathbb{N}}|f_n(x)-f_n(y)| \\
  	&\leq \sup_{n \in\mathbb{N}}\|f_n\|_{\operatorname{Lip}_0}\, d (x,y)=d (x,y), \quad \forall  x
  	, y\in \mathcal{M}.
  	\end{align*}
  	  We can now take $S$ as Lipschitz extension 
  	of $S_0$ to $\ell^\infty(\mathbb{N})$ and $\mathcal{M}_d=\ell^\infty(\mathbb{N})$ which completes the proof. 
  	\end{proof}
\begin{theorem}\label{METRICFRAMEEXISTS2}
	If 	$A:\mathcal{M} \to \mathcal{M}_d$ is bi-Lipschitz and there is a Lipschitz projection $P:\mathcal{M}_d \to A(\mathcal{M})$, then $\mathcal{M}$ admits a metric frame w.r.t. $\mathcal{M}_d$.
\end{theorem}
\begin{proof}
	Let $\{h_n\}_n$ be the sequence of coordinate functionals associated with $\mathcal{M}_d$. 	Define $f_n\coloneqq h_nA$ and $S \coloneqq A^{-1}P$. Then 
	\begin{align*}
	S(\{f_n(x)\}_n)=A^{-1}P(\{h_n(Ax)\}_n)=A^{-1}PAx=A^{-1}Ax=x, \quad \forall x \in \mathcal{M}.
	\end{align*}
	Hence  $(\{f_n\}_{n}, S)$ is  a metric frame  for $\mathcal{M}$ w.r.t. $\mathcal{M}_d$.
\end{proof}
  It is well-known that Mc-Schane extension theorem fails for complex valued Lipschitz functions. Thus we may ask whether we can take a complex
  sequence space in Theorem \ref{METRICFRAMEEXISTS}. It is possible for certain metric spaces due to the following theorem. 
  \begin{theorem}(Kirszbraun extension theorem) \cite{VALENTINEL}
  	Let $\mathcal{H}$ be a Hilbert  space and $\mathcal{M}_0$  be a nonempty subset
  	of $\mathcal{H}$. If $f_0:\mathcal{M}_0 \rightarrow \mathbb{K} $ is Lipschitz,
  	then there exists a Lipschitz function $f:\mathcal{H} \rightarrow \mathbb{K}
  	$ such that $f|{\mathcal{M}_0}=f_0$ and
  	$\operatorname{Lip}(f)=\operatorname{Lip}(f_0)$.
  \end{theorem}
  
  Following proposition shows that given a metric frame, we can generate other metric frames.	
  \begin{proposition}
  	Let $(\{f_n\}_{n}, S)$ be a metric frame  for $\mathcal{M}$ w.r.t. $\mathcal{M}_d$. If maps $A, B :\mathcal{M} \to \mathcal{M}$  are such that $A$ is 
  	bi-Lipschitz, $B$ is Lipschitz and $BA=I_\mathcal{M}$, then $(\{f_nA\}_{n}, BS)$ is a metric frame  for $\mathcal{M}$ w.r.t. $\mathcal{M}_d$. In particular, if $A :\mathcal{M} \to \mathcal{M}$   is 
  	bi-Lipschitz invertible, then $(\{f_nA\}_{n}, A^{-1}S)$ is a metric frame  for $\mathcal{M}$ w.r.t. $\mathcal{M}_d$.
  \end{proposition}
  \begin{proof}
  	Bi-Lipschitzness of $A$ tells that (ii) condition in Definition \ref{METRICBANACHFRAME} holds. Now by using $BA=I_\mathcal{M}$ we get $BS (\{f_nAx\}_{n})=BAx=x, \forall x \in \mathcal{M}.$
  \end{proof}
  Previous proposition not only helps to generate metric frames from metric frames but also from Banach frames. Since there are  large number of examples of Banach frames for a variety of Banach spaces, just by operating with  bi-Lipschitz invertible functions on subsets of it produces metric frames for that subset. Next we  characterize metric frames using Lipschitz functions. 
  \begin{theorem}\label{CHARLIPMETRIC}
  	Let $\{f_n\}_{n}$ be a metric $\mathcal{M}_d$-frame  for $\mathcal{M}$. Then the following are equivalent. 
  	\begin{enumerate}[\upshape(i)]
  		\item There exists a Lipschitz projection  $P:\mathcal{M}_d \to \theta_f(\mathcal{M})$. 
  		\item There exists a Lipschitz map $V:\mathcal{M}_d \to \mathcal{M}$ such that $V|_{\theta_f(\mathcal{M})}=\theta_f^{-1}$.
  		\item There exists a Lipschitz map $S:\mathcal{M}_d \to \mathcal{M}$ such that $(\{f_n\}_{n}, S)$ is  a metric frame  for $\mathcal{M}$ w.r.t. $\mathcal{M}_d$.
  	\end{enumerate}
  \end{theorem}
  \begin{proof}
  	(i)	$\Rightarrow$ (ii) Define $V\coloneqq \theta_f^{-1} P$. Then for $y=\theta_f(x),  x \in \mathcal{M}$ we get  $Vy=V\theta_f(x)=\theta_f^{-1} P\theta_f(x)=\theta_f^{-1}\theta_f(x)=\theta_f^{-1} y$.\\
  	(ii)	$\Rightarrow$ (i) Set $P\coloneqq \theta_fV$. Now $P^2=\theta_fV\theta_fV=\theta_fI_\mathcal{M}V=P$.\\
  	(ii)	$\Rightarrow$ (iii) Define $S\coloneqq V$. Then $S\{f_n(x)\}_n=S\theta_f (x)=V\theta_f (x)=x$, for all $x \in \mathcal{M}$. Hence  $(\{f_n\}_{n}, S)$ is  a metric frame  for $\mathcal{M}$ w.r.t. $\mathcal{M}_d$.\\
  	(iii)	$\Rightarrow$ (ii) Define $V\coloneqq S$. Then $V\theta_f (x)=S\theta_f (x)=S\{f_n(x)\}_n=x$, for all $x \in \mathcal{M}$.
  \end{proof}
  Now we turn onto the representation of elements using metric frames. Naturally, to deal with sums we  must look in Banach space structure. Following theorem can be compared with Theorem \ref{CASAZZASEPARABLECHARACTERIZATION}. 
  
  \begin{theorem}\label{PALL}
  	Let $\{f_n\}_{n}$ be a metric p-frame  for a Banach space $\mathcal{X}$. Assume that $f_n(0)=0$ for all $n$. Then the following are equivalent. 
  	\begin{enumerate}[\upshape(i)]
  		\item There exists a bounded linear map $V:\mathcal{M}_d \to \mathcal{X}$ such that $V|_{\theta_f(\mathcal{M})}=\theta_f^{-1}$.
  		\item There exists a bounded linear  map $S:\mathcal{M}_d \to \mathcal{X}$ such that $(\{f_n\}_{n}, S)$ is  a metric p-frame  for $\mathcal{X}$. 
  		\item There exists a sequence $\{\tau_n\}_{n}$ in  $\mathcal{X}$ such that $\sum_{n=1}^{\infty}c_n\tau_n$ converges for all $\{c_n\}_{n}\in \ell^p(\mathbb{N})$ and 
  	$
  			x=\sum_{n=1}^{\infty}f_n(x)\tau_n,  \forall x \in \mathcal{X}.
  		$
  		\item There exists a  q-Bessel sequence $\{\tau_n\}_{n}$ in  $\mathcal{X}\subseteq \mathcal{X}^{**}$ such that 
  	$	x=\sum_{n=1}^{\infty}f_n(x)\tau_n,  \forall x \in \mathcal{X}.
  		$
  		\item   There exists a  q-Bessel sequence $\{\tau_n\}_{n}$ in  $\mathcal{X}\subseteq \mathcal{X}^{**}$ such that 
  	$
  			f=\sum_{n=1}^{\infty}f(\tau_n)f_n,  \forall f \in \mathcal{X}^*.
  		$
  	\end{enumerate}	
  	In each of the cases (iv) and (v), $\{\tau_n\}_n$ is actually a q-frame for $\mathcal{X}^*$. 	 
  	\end{theorem}
  \begin{proof}
  	Proof of (i)	$\iff$ (ii) is similar to the proof of (ii)	$\iff$ (iii) in Theorem \ref{CHARLIPMETRIC}.\\
  	(iii)	$\Rightarrow$ (i) Given information tells that the map 
  	\begin{align*}
  		V: \ell^p(\mathbb{N}) \ni \{c_n\}_{n} \to \sum_{n=1}^{\infty}c_n\tau_n \in \mathcal{X}
  	\end{align*}
  	is well-defined. Banach-Steinhaus theorem now asserts that $V$ is bounded. Now for $y=\theta_f(x), $ $ x \in \mathcal{X}$ we get 
  	\begin{align*}
  		Vy=V\theta_f(x)=V(\{f_n(x)\}_{n})=\sum_{n=1}^{\infty}f_n(x)\tau_n=x=\theta_f^{-1} \theta_f(x)=\theta_f^{-1} y.
  	\end{align*} 
  	(i)	$\Rightarrow$ (iii) Let  $\{e_n\}_{n}$ be the standard Schauder basis for $\ell^p(\mathbb{N})$ and define $\tau_n \coloneqq Ve_n$, for all $n$. Since $V$ is bounded linear and $\sum_{n=1}^{\infty}c_ne_n$ converges for all $\{c_n\}_{n}\in \ell^p(\mathbb{N})$, it follows that $\sum_{n=1}^{\infty}c_n\tau_n$ converges for all $\{c_n\}_{n}\in \ell^p(\mathbb{N})$. Moreover, 
  	\begin{align*}
  		x=V\theta_f(x)=V(\{f_n(x)\}_{n})=\sum_{n=1}^{\infty}f_n(x)\tau_n, \quad \forall x \in \mathcal{X}.
  	\end{align*}
  	(iii)	$\iff$ (iv) By considering $\tau_n$ in $\mathcal{X}^{**}$ through James embedding and using Theorem \ref{CASAZZASEPARABLECHARACTERIZATION} we get that  $\{\tau_n\}_{n}$ is a q-Bessel sequence  in  $\mathcal{X}$ if and only if $\sum_{n=1}^{\infty}c_n\tau_n$ converges for all $\{c_n\}_{n}\in \ell^p(\mathbb{N})$. \\
  	(iv)	$\Rightarrow$ (v) Let $b$ be a Bessel bound for $\{\tau_n\}_{n}$. Then for all $f \in \mathcal{X}^*$ and $n\in \mathbb{N}$,
  	\begin{align*}
  	\left\|f-\sum_{k=1}^{n}f(\tau_k)f_k\right\|_{\operatorname{Lip}_0}&=\sup_{x, y \in \mathcal{X}, x\neq y} \frac{\left|\left(f-\sum_{k=1}^{n}f(\tau_k)f_k\right)(x)-\left(f-\sum_{k=1}^{n}f(\tau_k)f_k\right)(y)\right|}{\|x-y\|}\\
  	&=\sup_{x, y \in \mathcal{X}, x\neq y} \frac{\left|f\left(\sum_{k=1}^{\infty}f_k(x)\tau_k\right)-f\left(\sum_{k=1}^{\infty}f_k(y)\tau_k\right)-\sum_{k=1}^{\infty}f(\tau_k)(f_k(x)-f_k(y))\right|}{\|x-y\|}\\
  	&=\sup_{x, y \in \mathcal{X}, x\neq y} \frac{\left|\sum_{k=1}^{n}f(\tau_k)(f_k(x)-f_k(y))-\sum_{k=1}^{\infty}f(\tau_k)(f_k(x)-f_k(y))\right|}{\|x-y\|}\\
  	&=\sup_{x, y \in \mathcal{X}, x\neq y} \frac{\left|\sum_{k=n+1}^{\infty}f(\tau_k)(f_k(x)-f_k(y))\right|}{\|x-y\|}\\
  	&\leq \sup_{x, y \in \mathcal{X}, x\neq y} \frac{\left(\sum_{k=n+1}^{\infty}|f(\tau_k)|^q\right)^\frac{1}{q}\left(\sum_{k=n+1}^{\infty}|f_k(x)-f_k(y)|^p\right)^\frac{1}{p}}{\|x-y\|}\\
  	&\leq b \left(\sum_{k=n+1}^{\infty}|f_k(x)-f_k(y)|^p\right)^\frac{1}{p} \to 0 \text{ as } n \to \infty.
  		\end{align*}
  	(v)	$\Rightarrow$ (iv) Let $b$ be as in the last setting. Now given $x \in \mathcal{X}$ and $n\in \mathbb{N}$,
  	
  	\begin{align*}
  		\left\|x-\sum_{k=1}^{n}f_k(x)\tau_k\right\|&=\sup_{f\in \mathcal{X}^*, \|f\|=1} 	\left|f(x)-\sum_{k=1}^{n}f_k(x)f(\tau_k)\right|\\
  		&=\sup_{f\in \mathcal{X}^*, \|f\|=1} 	\left|\left(\sum_{k=1}^{\infty}f(\tau_k)f_k\right)(x)-\sum_{k=1}^{n}f_k(x)f(\tau_k)\right|\\
  		&=\sup_{f\in \mathcal{X}^*, \|f\|=1} 	\left|\sum_{k=n+1}^{\infty}f_k(x)f(\tau_k)\right|\\
  		&\leq \left(\sum_{k=n+1}^{\infty}|f(\tau_k)|^q\right)^\frac{1}{q}\left(\sum_{k=n+1}^{\infty}|f_k(x)-f_k(0)|^p\right)^\frac{1}{p}\\
  		&\leq b \left(\sum_{k=n+1}^{\infty}|f_k(x)|^p\right)^\frac{1}{p} \to 0 \text{ as } n \to \infty.
  	\end{align*}
  	Now we left with proving  that $\{\tau_n\}_{n}$ is a q-frame for $\mathcal{X}$. Assume (iv). Let $f \in \mathcal{X}^*.$ Then 
  	\begin{align*}
  	\|f\|&=\sup_{x\in \mathcal{X}, \|x\|=1} \left|f(x)\right|=\sup_{x\in \mathcal{X}, \|x\|=1} \left|f\left(\sum_{n=1}^{\infty}f_n(x)\tau_n\right)\right|\\
  	&=\sup_{x\in \mathcal{X}, \|x\|=1} \left|\sum_{n=1}^{\infty}f_n(x)f(\tau_n)\right|\leq b\left(\sum_{n=1}^{\infty}|f_n(x)|^p\right)^\frac{1}{p} .\\
  	\end{align*}
  	Since $f$ was arbitrary, the conclusion follows.
  	
  \end{proof}
  Theorem \ref{CASAZZASEPARABLECHARACTERIZATION} and Theorem \ref{PALL} suggest the following problem. 
  For which metric spaces and BK-spaces, does Theorem \ref{PALL} hold?  We next present a result which demands only reconstruction of elements using Lipschitz functions on Banach space and not frame conditions. First we record a result for this purpose. 
  \begin{lemma}
  	\cite{CASAZZACHRISTENSENSTOEVA}\label{ANOTHER}
  		Given  a Banach space $\mathcal{X}$ and a sequence $\{\tau_n\}_{n}$ of non-zero elements in $\mathcal{X}$, let 
  		\begin{align*}
  		\mathcal{Y}_d \coloneqq \left\{\{a_n\}_{n}:\sum_{n=1}^\infty a_n \tau_n \text{ converges in }   \mathcal{X}\right\}. 
  		\end{align*}
  		Then $\mathcal{Y}_d$ is a Banach space w.r.t. the norm 
  		\begin{align*}
  		\| \{a_n\}_{n}\|\coloneqq \sup_{m }\left\|\sum_{n=1}^m a_n \tau_n\right\|.
  		\end{align*}
  		Further, the canonical unit vectors form a Schauder basis for $\mathcal{Y}_d$.
  	
  \end{lemma}
  \begin{theorem}
  Let $\mathcal{X}$ be a Banach space and   $\{f_n\}_{n}$ be a sequence in   $\operatorname{Lip}_0(\mathcal{X}, \mathbb{K})$.   Then the following are equivalent. 
  \begin{enumerate}[\upshape(i)]
  	\item There exists a  sequence $\{\tau_n\}_{n}$ in  $\mathcal{X}$ such that 
  	$
  	x=\sum_{n=1}^{\infty}f_n(x)\tau_n,  \forall x \in \mathcal{X}.
  	$
  	\item Let $\{\tau_n\}_{n}$ be a sequence in $\mathcal{X}$ and  define $S_n(x)\coloneqq \sum_{k=1}^{n}f_k(x)\tau_k$, $\forall x \in \mathcal{X}$, for each $n \in \mathbb{N}$. Then $\sup_{n \in\mathbb{N}}\|S_n\|_{\operatorname{Lip}_0} <\infty$ and    there exist a BK-space $\mathcal{M}_d$ and a bounded linear  map $S:\mathcal{M}_d \to \mathcal{M}$ such that $(\{f_n\}_{n}, S)$ is  a metric frame  for $\mathcal{X}$ w.r.t. $\mathcal{M}_d$. 
  \end{enumerate}
Further, a choice for $\tau_n$ is $\tau_n=Se_n$ for each $n \in \mathbb{N}$, where $\{e_n\}_{n}$ is the standard Schauder basis for $\ell^p(\mathbb{N})$.
  \end{theorem}
  \begin{proof}
  	(ii)	$\Rightarrow$ (i) This follows from Theorem \ref{PALL}.\\
  	(i)	$(\Rightarrow)$ (ii) We give an argument which is similar to   the arguments in \cite{CASAZZACHRISTENSENSTOEVA}. Define $A\coloneqq \{n \in \mathbb{N}: \tau_n=0\}$ and $B\coloneqq \mathbb{N} \setminus A$. Let $c_0(A) $ be the space of sequences converging to zero, indexed by $A$, equipped with sup-norm.  Let $\{e_n\}_{n\in A}$ be the canonical Schauder basis for $c_0(A) $. Since the norm is sup-norm, it easily follows that $\{\frac{1}{n(\|f_n\|_{\operatorname{Lip}_0}+1)}e_n\}_{n\in A}$ is also a Schauder basis for $c_0(A) $. Define 
  	\begin{align*}
  	\mathcal{Z}_d\coloneqq \left\{\{c_n\}_{n\in A}: \sum_{n\in A}\frac{c_n}{n(\|f_n\|_{\operatorname{Lip}_0}+1)}e_n \text{ converges in } A\right\}.
  	\end{align*}
  	We equip $\mathcal{Z}_d $ with the norm 
  	\begin{align*}
  	\|\{c_n\}_{n\in A}\|_{\mathcal{Z}_d}\coloneqq \left\|\frac{c_n}{n(\|f_n\|_{\operatorname{Lip}_0}+1)}\right\|_{c_0(A)}=\sup _{n\in A}\left|\frac{c_n}{n(\|f_n\|_{\operatorname{Lip}_0}+1)}\right|.
  	\end{align*}
  	Then $\{e_n\}_{n\in A}$ is a Schauder basis for $\mathcal{Z}_d $. Clearly $\mathcal{Z}_d $ is a BK-space. Let $\mathcal{Y}_d $ be as defined in  Lemma \ref{ANOTHER}, for the index set $B$. Now set $\mathcal{M}_d \coloneqq \mathcal{Y}_d \oplus \mathcal{Z}_d $ equipped with norm $\|y \oplus z \|_{\mathcal{M}_d}\coloneqq \|y\|_{\mathcal{Y}_d} +\|z\|_{\mathcal{Z}_d}$. It then follows that, for each $x \in \mathcal{X}$, $\{f_n(x)\}_{n\in B}\oplus \{f_n(x)\}_{n\in A} \in \mathcal{M}_d$. We next show that $\{f_n\}_{n}$ is a metric  $\mathcal{M}_d$-frame for $\mathcal{X}$. Let $x, y \in \mathcal{X}$. Then 
  	\begin{align*}
  	\|x-y\|&=\left\|\sum_{n=1}^{\infty}(f_n(x)-f_n(y))\tau_n\right\|=\lim_{n\to\infty}\left\|\sum_{k=1}^{n}(f_k(x)-f_k(y))\tau_k\right\|\\
  	&\leq \sup _{n\in \mathbb{N}}\left\|\sum_{k=1}^{n}(f_k(x)-f_k(y))\tau_k\right\|=\sup _{n\in B}\left\|\sum_{k=1}^{n}(f_k(x)-f_k(y))\tau_k\right\|\\
  	&=\|\{f_n(x)-f_n(y)\}_{n\in B}\|_{\mathcal{Y}_d}\\
  	&\leq \|\{f_n(x)-f_n(y)\}_{n\in B}\|_{\mathcal{Y}_d}+\|\{f_n(x)-f_n(y)\}_{n\in A}\|_{\mathcal{Z}_d}\\
  	&= \|\{f_n(x)-f_n(y)\}_{n\in B}\oplus \{f_n(x)-f_n(y)\}_{n\in A}\|_{\mathcal{M}_d}
  	\end{align*}
  	and 
  	\begin{align*}
  	&\|\{f_n(x)-f_n(y)\}_{n\in B}\oplus \{f_n(x)-f_n(y)\}_{n\in A}\|_{\mathcal{M}_d}\\
  	&=\|\{f_n(x)-f_n(y)\}_{n\in B}\|_{\mathcal{Y}_d}+\|\{f_n(x)-f_n(y)\}_{n\in A}\|_{\mathcal{Z}_d}\\
  	&=\sup _{n\in B}\left\|\sum_{k=1}^{n}(f_k(x)-f_k(y))\tau_k\right\|+\sup _{n\in A}\left|\frac{f_n(x)-f_n(y)}{n(\|f_n\|_{\operatorname{Lip}_0}+1)}\right|\\
  	&=\sup _{n\in B}\left\|S_n(x)-S_n(y)\right\|+\sup _{n\in A}\left|\frac{f_n(x)-f_n(y)}{n(\|f_n\|_{\operatorname{Lip}_0}+1)}\right|\\
  	&\leq \sup_{n \in B}\|S_n\|_{\operatorname{Lip}_0}\|x-y\|+\sup _{n\in A}\frac{\|f_n\|_{\operatorname{Lip}_0}\|x-y\|}{n(\|f_n\|_{\operatorname{Lip}_0}+1)}\\
  	&\leq\left(\sup_{n \in B}\|S_n\|_{\operatorname{Lip}_0}+1\right)\|x-y\|.
  	\end{align*}
  	We now define 
  	\begin{align*}
  	S:\mathcal{M}_d \ni \{a_n\}_{n\in B}\oplus \{b_n\}_{n\in A}\mapsto \sum_{n\in B}a_n\tau_n \in \mathcal{X}.
  	\end{align*}
   Clearly $S$ is linear. Proof completes if we show that $S$ is bounded. This follows from the following calculation.
   
  \begin{align*}
  \|S(\{a_n\}_{n\in B}\oplus \{b_n\}_{n\in A})\|&=\left\|\sum_{n\in B}a_n\tau_n\right\|\leq \sup _{n\in B}\left\|\sum_{k=1}^n a_k\tau_k\right\|\\
  &=\|\{a_n\}_{n\in B}\|_{\mathcal{Y}_d} \leq \|\{a_n\}_{n\in B}\oplus \{b_n\}_{n\in A}\|_{\mathcal{M}_d}.
  \end{align*}
  \end{proof}
  
  In the study of Lipschitz functions it is natural to shift from metric space to the setting of Banach spaces and use functional analysis tools on Banach spaces. This is achieved through the following theorem. 
  \begin{theorem}\cite{WEAVER, KALTON}\label{POINTEDSPLITS}
  	Let $(\mathcal{M},0)$ be a pointed metric space. Then there exists a Banach space  $\mathcal{F}(\mathcal{M})$ and an isometric embedding $e:\mathcal{M} \to \mathcal{F}(\mathcal{M})$ satisfying
  	the following  universal property: for each Banach space $\mathcal{X}$ and each $f \in \operatorname{Lip}_0(\mathcal{M}, \mathcal{X})$, there is a unique bounded linear operator 
  	$T_f :\mathcal{F}(\mathcal{M})\to \mathcal{X} $ such that $T_fe=f$, i.e., the following diagram commutes.
  	
  	\begin{center}
  		\[
  		\begin{tikzcd}
  		 \mathcal{M} \arrow[d,"e"] \arrow[dr,"f"]\\
  		 \mathcal{F}(\mathcal{M}) \arrow[r,"T_f"] & \mathcal{X}
  		\end{tikzcd}
  	  \]
  	\end{center}
  	Further, $\|T_f\|=\|f\|_{\operatorname{Lip}_0}$. This property characterizes the pair $(\mathcal{F}(\mathcal{M}),e)$ uniquely upto isometric isomorphism. Moreover, the map 
  	$\operatorname{Lip}_0(\mathcal{M}, \mathcal{X})\ni f \mapsto T_f \in \mathcal{B}(\mathcal{F}(\mathcal{M}),\mathcal{X}) $ (space of bounded linear operators from $\mathcal{F}(\mathcal{M})$ to $\mathcal{X} $) is an isometric isomorphism.
  \end{theorem}

  The space $\mathcal{F}(\mathcal{M})$ is known as Arens-Eells space or  Lipschitz-free 	Banach space (see \cite{GODEFROYSURVEY}). Theorem \ref{POINTEDSPLITS} tells that in order to ``find" the space $\operatorname{Lip}_0(\mathcal{M}, \mathcal{X})$, we can find first $\mathcal{F}(\mathcal{M})$ and then $\mathcal{B}(\mathcal{F}(\mathcal{M}),\mathcal{X}) $. In particular, $\operatorname{Lip}_0(\mathcal{M}, \mathbb{K})$ is isometrically isomorphic to  $\mathcal{F}(\mathcal{M})^*$. For this reason $\mathcal{F}(\mathcal{M})$ is also called as predual of metric space $\mathcal{M}$. As examples, it is known that $\mathcal{F}(\mathbb{N})\equiv \ell^1(\mathbb{N})$, $\mathcal{F}(\mathbb{R})\equiv \mathcal{L}^1(\mathbb{R})$.  The bounded linear  operator $T_f$ is called as linearizaton of $f$.  \\
  
  Using Theorem \ref{POINTEDSPLITS} we derive the following result which tells that given a metric frame for a metric space we can get a metric frame using linear functionals for a subset of the Banach space.
  \begin{theorem}\label{LIPIFFLINEAR}
  	Let $\{f_n\}_{n}$  be a sequence in $ \operatorname{Lip}_0(\mathcal{M}, \mathbb{K})$. For  each $n\in \mathbb{N}$, let  $T_{f_n}$ be linearization of $f_n$. Let $e$ and $\mathcal{F}(\mathcal{M})$ 
  	be as in Theorem \ref{POINTEDSPLITS}. Then $\{f_n\}_{n}$ is a 
  	metric frame  for $\mathcal{M}$ w.r.t. $\mathcal{M}_d$ with bounds $a$ and $b$ if and only if $\{T_{f_n}\}_{n}$ is a 
  	metric frame  for $e(\mathcal{M})$ w.r.t. $\mathcal{M}_d$ with bounds $a$ and $b$. In particular, $(\{f_n\}_{n}, S)$ is a metric frame  for $\mathcal{M}$ w.r.t. $\mathcal{M}_d$ if and only if $(\{T_{f_n}\}_{n}, eS)$ is a metric frame  for $e(\mathcal{M})$ w.r.t. $\mathcal{M}_d$.
  \end{theorem}
  \begin{proof}
  	$(\Rightarrow)$ Let $u,v \in e(\mathcal{M})$. Then $u=e(x), v=e(y)$, for some $x, y \in \mathcal{M}$. Now using the fact that $e$ is an isometry,
  	\begin{align*}
  		a\|u-v\|&=a\|e(x)-e(y)\|=a\,d(x,y)\leq \|\{f_n(x)-f_n(y)\}_n\|\\
  		&=\|\{(T_{f_n}e)(x)-(T_{f_n}e)(y)\|=\|\{T_{f_n}(e(x))-T_{f_n}(e(y))\}_n\|\\
  		&=\|\{T_{f_n}(u)-T_{f_n}(v)\}_n\|\leq b \,d(x,y)=b\|e(x)-e(y)\|=b\|u-v\|.
  	\end{align*}
  	$(\Leftarrow)$ Let $x, y \in \mathcal{M}$. Then $e(x), e(y) \in e(\mathcal{M})$. Hence 
  	
  	\begin{align*}
  		a\,d(x,y)&=a\|e(x)-e(y)\|\leq \|\{T_{f_n}(e(x))-T_{f_n}(e(y))\}_n\|\\
  		&=\|\{f_n(x)-f_n(y)\}_n\|\leq b\|e(x)-e(y)\|=b\,d(x,y).
  	\end{align*}
  	Since $x,y$ were arbitrary, the result follows.
  \end{proof}
   We end this paper by presenting some stability results. These are important as it says that sequences which are close to metric frames are again metric frames. On the other hand, it asserts that if we perturb a metric frame we again get metric frame.
 \begin{theorem}\label{FIRSTPERTURB}
  	Let $\{f_n\}_{n}$ be a
  	metric frame  for $\mathcal{M}$ w.r.t. $\ell^p(\mathbb{N})$ with bounds $a$ and $b$. Let $\{g_n\}_{n}$ be a sequence in $\operatorname{Lip}(\mathcal{M}, \mathbb{K})$ satisfying the following.
  	\begin{enumerate}[\upshape(i)]
  		\item There exist $\alpha, \beta, \gamma \geq 0$ such that $\beta<1$.
  		\item $\alpha<1$, $\gamma<(1-\alpha)a$.
  		\item For all $x, y \in  \mathcal{M}$, 
  		\begin{align}\label{PERINEQUA}
  			\left(\sum_{n=1}^m|(f_n-g_n)(x)-(f_n-g_n)(y)|^p\right)^\frac{1}{p}&\leq \alpha \left(\sum_{n=1}^m|f_n(x)-f_n(y)|^p\right)^\frac{1}{p}\\
  			&+
  			\beta \left(\sum_{n=1}^m|g_n(x)-g_n(y)|^p\right)^\frac{1}{p}+\gamma \,d(x,y), \quad m=1, 2,\dots \nonumber.
  		\end{align}
  	\end{enumerate}
  	Then $\{g_n\}_{n}$ is a
  	metric frame  for $\mathcal{M}$ w.r.t. $\ell^p(\mathbb{N})$ with bounds
  $
  		\frac{((1-\alpha)a-\gamma)}{1+\beta}$ and  $ \frac{((1+\alpha)b+\gamma)}{1-\beta}.
  	$
  \end{theorem}
  \begin{proof}
  Using Minkowski's inequality and Inequality (\ref{PERINEQUA}), we get, for all $x, y \in  \mathcal{M}$ and $m\in \mathbb{N}$,
  	\begin{align*}
  	\left(\sum_{n=1}^m|g_n(x)-g_n(y)|^p\right)^\frac{1}{p}&\leq \left(\sum_{n=1}^m|(f_n-g_n)(x)-(f_n-g_n)(y)|^p\right)^\frac{1}{p}+ \left(\sum_{n=1}^m|f_n(x)-f_n(y)|^p\right)^\frac{1}{p}\\
  	&\leq (1+\alpha) \left(\sum_{n=1}^m|f_n(x)-f_n(y)|^p\right)^\frac{1}{p}+\beta \left(\sum_{n=1}^m|g_n(x)-g_n(y)|^p\right)^\frac{1}{p}+\gamma \,d(x,y)\\
  	\end{align*}
  	which implies 
  	\begin{align*}
  	(1-\beta)\left(\sum_{n=1}^m|g_n(x)-g_n(y)|^p\right)^\frac{1}{p}\leq (1+\alpha)\left(\sum_{n=1}^m|f_n(x)-f_n(y)|^p\right)^\frac{1}{p}+\gamma  \,d(x,y), \quad \forall x, y \in  \mathcal{M}.\\
  	\end{align*}
  	Since the sum $\sum_{n=1}^\infty|f_n(x)-f_n(y)|^p$ converges, $\sum_{n=1}^\infty|g_n(x)-g_n(y)|^p$ will also converge. Inequality (\ref{PERINEQUA}) now gives 
  		\begin{align}\label{PERINEQUA2}
  	\left(\sum_{n=1}^\infty|(f_n-g_n)(x)-(f_n-g_n)(y)|^p\right)^\frac{1}{p}&\leq \alpha \left(\sum_{n=1}^\infty|f_n(x)-f_n(y)|^p\right)^\frac{1}{p}\nonumber\\
  	&+
  	\beta \left(\sum_{n=1}^\infty|g_n(x)-g_n(y)|^p\right)^\frac{1}{p}+\gamma \,d(x,y) .
  	\end{align}
  	By doing a similar calculation and using Inequality (\ref{PERINEQUA2}) we get  for all $x, y \in  \mathcal{M}$,
  	
  	\begin{align*}
  		\left(\sum_{n=1}^\infty|g_n(x)-g_n(y)|^p\right)^\frac{1}{p}&\leq \left(\sum_{n=1}^\infty|(f_n-g_n)(x)-(f_n-g_n)(y)|^p\right)^\frac{1}{p}+ \left(\sum_{n=1}^\infty|f_n(x)-f_n(y)|^p\right)^\frac{1}{p}\\
  		&\leq (1+\alpha) \left(\sum_{n=1}^\infty|f_n(x)-f_n(y)|^p\right)^\frac{1}{p}+\beta \left(\sum_{n=1}^\infty|g_n(x)-g_n(y)|^p\right)^\frac{1}{p}+\gamma \,d(x,y)\\
  		&\leq (1+\alpha)b  \,d(x,y)+\beta \left(\sum_{n=1}^\infty|g_n(x)-g_n(y)|^p\right)^\frac{1}{p}+\gamma \,d(x,y)\\
  		&=((1+\alpha)b+\gamma)  \,d(x,y)+\beta \left(\sum_{n=1}^\infty|g_n(x)-g_n(y)|^p\right)^\frac{1}{p}
  	\end{align*}
  	which gives
  	\begin{align*}
  		(1-\beta)\left(\sum_{n=1}^\infty|g_n(x)-g_n(y)|^p\right)^\frac{1}{p}\leq ((1+\alpha)b+\gamma)  \,d(x,y), \quad \forall x, y \in  \mathcal{M}\\
  		\text{i.e.,} \left(\sum_{n=1}^\infty|g_n(x)-g_n(y)|^p\right)^\frac{1}{p}\leq \frac{((1+\alpha)b+\gamma)}{1-\beta}\,d(x,y), \quad \forall x, y \in  \mathcal{M}.
  	\end{align*}
  	Hence we obtained upper frame bound for $\{g_n\}_n$. For lower frame bound, let  $x, y \in  \mathcal{M}$. Then 
  	\begin{align*}
  		\left(\sum_{n=1}^\infty|f_n(x)-f_n(y)|^p\right)^\frac{1}{p}&\leq \left(\sum_{n=1}^\infty|(f_n-g_n)(x)-(f_n-g_n)(y)|^p\right)^\frac{1}{p}+ \left(\sum_{n=1}^\infty|g_n(x)-g_n(y)|^p\right)^\frac{1}{p}\\
  		&\leq \alpha \left(\sum_{n=1}^\infty|f_n(x)-f_n(y)|^p\right)^\frac{1}{p}+(1+\beta)\left(\sum_{n=1}^\infty|g_n(x)-g_n(y)|^p\right)^\frac{1}{p}+\gamma \,d(x,y)
  	\end{align*}
  	which implies 
  	\begin{align*}
  		(1-\alpha)a\,d(x,y)&\leq (1-\alpha)\left(\sum_{n=1}^\infty|f_n(x)-f_n(y)|^p\right)^\frac{1}{p}\\
  		&\leq (1+\beta)\left(\sum_{n=1}^\infty|g_n(x)-g_n(y)|^p\right)^\frac{1}{p}+\gamma \,d(x,y), \quad \forall x, y \in  \mathcal{M}\\
  		\text{i.e.,} ~\frac{((1-\alpha)a-\gamma)}{1+\beta} &\leq \left(\sum_{n=1}^\infty|g_n(x)-g_n(y)|^p\right)^\frac{1}{p}, \quad \forall x, y \in  \mathcal{M}.
  	\end{align*}
  \end{proof}
In the theory of Hilbert spaces, it is known that sequences which are quadratically close to frames are again frames (see \cite{PER1995}). Using Theorem \ref{FIRSTPERTURB} we obtain a similar result for metric frames.	
  \begin{corollary}
  		Let $\{f_n\}_{n}$ be a
  	metric frame  for $\mathcal{M}$ w.r.t. $\ell^p(\mathbb{N})$ with bounds $a$ and $b$. Let $\{g_n\}_{n}$ be a sequence in $\operatorname{Lip}(\mathcal{M}, \mathbb{K})$ such that 
  	\begin{align*}
  	r\coloneqq \left(\sum_{n=1}^\infty \operatorname{Lip}(f_n-g_n)^p\right)^\frac{1}{p} <a.
  	\end{align*}
  	Then $\{g_n\}_{n}$ is a
  	metric frame  for $\mathcal{M}$ w.r.t. $\ell^p(\mathbb{N})$ with bounds $a-r$ and $b+r$.
  \end{corollary}
\begin{proof}
Define $\alpha\coloneqq 0$, $\beta\coloneqq 0$ and $\gamma\coloneqq r$. Then
	for all $x, y \in  \mathcal{M}$, 
	\begin{align*}
	\left(\sum_{n=1}^\infty|(f_n-g_n)(x)-(f_n-g_n)(y)|^p\right)^\frac{1}{p}&\leq \left(\sum_{n=1}^\infty \operatorname{Lip}(f_n-g_n)^p\,d(x,y)^p\right)^\frac{1}{p}\\
	&=\left(\sum_{n=1}^\infty \operatorname{Lip}(f_n-g_n)^p\right)^\frac{1}{p}\,d(x,y)=r\,d(x,y)\\
	&=\alpha \left(\sum_{n=1}^\infty|f_n(x)-f_n(y)|^p\right)^\frac{1}{p}\nonumber
	+
	\beta \left(\sum_{n=1}^\infty|g_n(x)-g_n(y)|^p\right)^\frac{1}{p}+\gamma \,d(x,y).
	\end{align*}
Thus the hypothesis in Theorem \ref{FIRSTPERTURB} holds. Hence the corollary.
\end{proof}
An inspection of proof of Theorem \ref{FIRSTPERTURB} gives the following result easily.
  \begin{corollary}
  	Let $\{f_n\}_{n}$ be a
  	metric Bessel sequence   for $\mathcal{M}$ w.r.t. $\ell^p(\mathbb{N})$ with bound $b$. Let $\{g_n\}_{n}$ be a sequence in $\operatorname{Lip}(\mathcal{M}, \mathbb{K})$ satisfying the following.
  	\begin{enumerate}[\upshape(i)]
  		\item There exist $\alpha, \beta, \gamma \geq 0$ such that $\beta<1$.
  		\item For all $x, y \in  \mathcal{M}$,
  		\begin{align*}
  			\left(\sum_{n=1}^m|(f_n-g_n)(x)-(f_n-g_n)(y)|^p\right)^\frac{1}{p}&\leq \alpha \left(\sum_{n=1}^m|f_n(x)-f_n(y)|^p\right)^\frac{1}{p}\\
  			&+
  			\beta \left(\sum_{n=1}^m|g_n(x)-g_n(y)|^p\right)^\frac{1}{p}+\gamma \,d(x,y), \quad m=1,\dots .
  		\end{align*}
  	\end{enumerate}
  	Then $\{g_n\}_{n}$ is a
  	metric Bessel sequence for $\mathcal{M}$ w.r.t. $\ell^p(\mathbb{N})$ with bound $	\frac{((1+\alpha)b+\gamma)}{1-\beta}.$
  \end{corollary}
 We next derive a stability result in which we perturb the Lipschitz functions and then derive the existence of reconstruction operator.  This is motivated from a result in \cite{CHRISTENSENHEIL}.	
  
  \begin{theorem}\label{STABILITYMA}
  	Let $(\{f_n\}_{n}, S)$ be a
  	metric frame for a Banach space    $\mathcal{X}$ w.r.t. $\mathcal{M}_d$. Assume that
  	$f_n(0)=0,$ for all $n \in \mathbb{N}$, and $S0=0$. Let $\{g_n\}_{n}$ be a collection in 
  	$\operatorname{Lip}_0(\mathcal{X}, \mathbb{K})$ satisfying the following.
  	\begin{enumerate}[\upshape(i)]
  		\item There exist $\alpha, \gamma\geq 0$ such that 
  		\begin{align}\label{MFPER}
  			\|\{(f_n-g_n)(x)-(f_n-g_n)(y)\}_n\|\leq \alpha \|\{f_n(x)-f_n(y)\}_n\|+\gamma \|x-y\|, \quad \forall x, y \in  \mathcal{X}.
  		\end{align}
  		\item $\alpha \|\theta_f\|_{\operatorname{Lip}_0}+\gamma\leq \|S\|_{\operatorname{Lip}_0}^{-1}.$
  	\end{enumerate}
  	Then there exists a reconstruction Lipschitz operator $T$ such that $(\{f_n\}_{n}, T)$ is  a
  	metric frame for     $\mathcal{X}$ w.r.t. $\mathcal{M}_d$ with bounds 
  	$
  		\|S\|_{\operatorname{Lip}_0}^{-1}-(\alpha\|\theta_f\|_{\operatorname{Lip}_0}+\gamma) $ and $ 
  		\|\theta_f\|_{\operatorname{Lip}_0}+(\alpha\|\theta_f\|_{\operatorname{Lip}_0}+\gamma).$
  	
  \end{theorem}
  \begin{proof}
  	Let $ x\in \mathcal{X}$. Since $g_n(0)=0$ and $f_n(0)=0$ for all $n \in \mathbb{N}$, using Inequality (\ref{MFPER}),
  	\begin{align*}
  		\|\{g_n(x)\}_n\| &\leq \|\{(f_n-g_n)(x)\}_n\|+\|\{f_n(x)\}_n\|\\
  		&\leq (\alpha+1)\|\{f_n(x)\}_n\|+\gamma \|x\|.
  	\end{align*}
  	Therefore if we define $\theta_g:\mathcal{X} \ni x \mapsto \{g_n(x)\}_{n} \in\mathcal{M}_d$, then this map is well-defined. Again 
  	using Inequality (\ref{MFPER}), we show that $\theta_g$ is Lipschitz. For $x, y \in  \mathcal{X}$,
  	
  	\begin{align*}
  		\|\theta_gx-\theta_gy\|&=\|\{g_n(x)-g_n(y)\}_n\|=\|\{-g_n(x)+g_n(y)\}_n\| \\
  		&\leq \|\{(f_n-g_n)(x)-(f_n-g_n)(y)\}_n\|+\|\{f_n(x)-f_n(y)\}_n\|\\
  		&\leq (1+\alpha)\|\{f_n(x)-f_n(y)\}_n\|+\gamma \|x-y\|=(1+\alpha)\|\theta_fx-\theta_fy\|+\gamma \|x-y\|\\
  		&\leq (1+\alpha) \|\theta_f\|_{\operatorname{Lip}_0}\|x-y\|+\gamma \|x-y\|
  		=((1+\alpha) \|\theta_f\|_{\operatorname{Lip}_0}+\gamma)\|x-y\|.
  	\end{align*}
  	Thus $\|\theta_g\|_{\operatorname{Lip}_0}\leq (1+\alpha) \|\theta_f\|_{\operatorname{Lip}_0}+\gamma$. Previous calculation
  	also tells that upper frame bound is $((1+\alpha) \|\theta_f\|_{\operatorname{Lip}_0}+\gamma)$. We see further that  
  	Inequality (\ref{MFPER}) can be written as 
  	\begin{align}\label{THETAFTHETAG}
  		\|(\theta_f-\theta_g)x-(\theta_f-\theta_g)y\|&\leq \alpha \|\theta_fx-\theta_fy\|+\gamma \|x-y\|\nonumber\\
  		&\leq (\alpha \|\theta_f\|_{\operatorname{Lip}_0}+\gamma)\|x-y\|, \quad \forall x, y \in  \mathcal{X}.
  	\end{align}
  	Now noting 
  	$S\theta_f=I_\mathcal{X}$ and using Inequality (\ref{THETAFTHETAG}) we see that 
  	\begin{align*}
  		\|I_\mathcal{X}-S\theta_g\|_{\operatorname{Lip}_0}&=\|S\theta_f-S\theta_g\|_{\operatorname{Lip}_0}\\
  		&\leq \|S\|_{\operatorname{Lip}_0}\|\theta_f-\theta_g\|_{\operatorname{Lip}_0}\\
  		&\leq \|S\|_{\operatorname{Lip}_0}(\alpha \|\theta_f\|_{\operatorname{Lip}_0}+\gamma)<1.
  	\end{align*}
  	Since $\operatorname{Lip}_0(\mathcal{X})$ is a unital Banach algebra (Theorem \ref{BANACHALGEBRA}), last inequality tells that $S\theta_g$ is invertible and its inverse 
  	is also Lipschitz operator and 
  	\begin{align*}
  		\|(S\theta_g)^{-1}\|_{\operatorname{Lip}_0}\leq 
  		\frac{1}{1-\|S\|_{\operatorname{Lip}_0}(\alpha \|\theta_f\|_{\operatorname{Lip}_0}+\gamma)}.
  	\end{align*}
  	Define $T\coloneqq (S\theta_g)^{-1} S$. Then $T\theta_g=I_\mathcal{X}$ and 
  	\begin{align*}
  		\|x-y\|&=\|T\theta_gx-T\theta_gy\|\leq \|T\|_{\operatorname{Lip}_0}\|\theta_gx-\theta_gy\|\\
  		&\leq \frac{1}{1-\|S\|_{\operatorname{Lip}_0}(\alpha \|\theta_f\|_{\operatorname{Lip}_0}+\gamma)}\|\theta_gx-\theta_gy\|, 
  		\quad \forall x, y \in  \mathcal{X}
  	\end{align*}
  	which gives the lower bound stated in theorem.
  \end{proof}

  \begin{corollary}
  	Let $(\{f_n\}_{n}, S)$ be a
  	metric Bessel sequence for a Banach space    $\mathcal{X}$ w.r.t. $\mathcal{M}_d$. Assume that
  	$f_n(0)=0,$ for all $n \in \mathbb{N}$, and $S0=0$. Let $\{g_n\}_{n}$ be a collection in 
  	$\operatorname{Lip}_0(\mathcal{X}, \mathbb{K})$ satisfying the following.
  	There exist $\alpha, \gamma\geq 0$ such that 
  	\begin{align*}
  		\|\{(f_n-g_n)(x)-(f_n-g_n)(y)\}_n\|\leq \alpha \|\{f_n(x)-f_n(y)\}_n\|+\gamma \|x-y\|, \quad \forall x, y \in  \mathcal{X}.
  	\end{align*}
  	Then there exists a reconstruction Lipschitz operator $T$ such that $(\{f_n\}_{n}, T)$ is  a
  	metric Bessel sequence for      $\mathcal{X}$ w.r.t. $\mathcal{M}_d$ with bound 
  	$\|\theta_f\|_{\operatorname{Lip}_0}+(\alpha\|\theta_f\|_{\operatorname{Lip}_0}+\gamma)$.
  \end{corollary}

 \section{Acknowledgements}
  First author thanks National Institute of Technology (NITK) Surathkal for
 financial assistance.

 \bibliographystyle{plain}
 \bibliography{reference.bib}

\end{document}